\documentclass[11 pt,reqno]{amsart}
\usepackage{fullpage}

\usepackage{amsmath}
\usepackage{amsfonts}
\usepackage{amssymb}
\usepackage{amsthm}
\usepackage{mathtools}
\usepackage{enumitem}
\usepackage{esint}
\usepackage{hyperref}

\makeatletter
\newtheorem*{rep@theorem}{\rep@title}
\newcommand{\newreptheorem}[2]{%
\newenvironment{rep#1}[1]{%
 \def\rep@title{#2 \ref{##1}}%
 \begin{rep@theorem}}%
 {\end{rep@theorem}}}
\makeatother

\newtheorem{theorem}{Theorem}[section]
\newreptheorem{theorem}{Theorem}
\newtheorem*{theorem*}{Theorem}
\newtheorem{corollary}[theorem]{Corollary}
\newtheorem{conjecture}[theorem]{Conjecture}

\newtheorem{lemma}[theorem]{Lemma}

\theoremstyle{definition}
\newtheorem{definition}[theorem]{Definition}
\newtheorem{example}[theorem]{Example}

\theoremstyle{remark}
\newtheorem{remark}[theorem]{Remark}

\def \R {\mathbb{R}}
\def \E {\mathbb{E}}
\def \P {\mathbb{P}}

\def \Z {\mathbb{Z}}
\def \g {\mathfrak{g}}
\def \comp {\circ}
\def \sbs {\subseteq}

\def \cross {\times}

\def \eps {\epsilon}

\def \ones {\mathsf{1}}
\def \ad {\text{ad}}

\usepackage{graphicx}

\numberwithin{equation}{section}

\begin{document}

\title{Estimates on the Markov Convexity of Carnot Groups and Quantitative Nonembeddability}


\author{Chris Gartland}
\thanks{Thanks to Jeremy Tyson for helpful comments in the preparation of this article and to Assaf Naor for suggesting the (non)embeddability corollaries of the main theorems.}






\begin{abstract}
We show that every graded nilpotent Lie group $G$ of step $r$, equipped with a left invariant metric homogeneous with respect to the dilations induced by the grading, (this includes all Carnot groups with Carnot-Caratheodory metric) is Markov $p$-convex for all $p \in [2r,\infty)$. We also show that this is sharp whenever $G$ is a Carnot group with $r \leq 3$, a free Carnot group, or a jet space group; such groups are not Markov $p$-convex for any $p \in (0,2r)$. This continues a line of research started by Li who proved this sharp result when $G$ is the Heisenberg group. As corollaries, we obtain new estimates on the non-biLipschitz embeddability of some finitely generated nilpotent groups into nilpotent Lie groups of lower step. Sharp estimates of this type are known when the domain is the Heisenberg group and the target is a uniformly convex Banach space or $L^1$, but not when the target is a nonabelian nilpotent group.
\end{abstract}

\maketitle


\tableofcontents

\section{Introduction}
\subsection{Background}
In \cite{Ribe}, Ribe showed that if two Banach spaces $E,F$ are uniformly homeomorphic, then they are mutually finitely representable; there exists a $\lambda < \infty$ such that for any finitely dimensional subspace $E_1$ of $E$, there is a subspace $F_1$ of $F$ whose Banach-Mazur distance from $E_1$ is at most $\lambda$. Properties of Banach spaces that are preserved under mutual finite representability are called \emph{local}, and many classical properties such as type, cotype, superreflexivity, and $p$-convexity are local. Recall that a Banach space is said to be \emph{$p$-convex} for some $p \geq 2$ if there exists an equivalent norm $\|\cdot\|$ and $K < \infty$ such that for every $\eps \in [0,2]$,
$$\sup\{\|(x+y)/2\|: \|x\|,\|y\| \leq 1, \|x-y\| \geq \eps\} \leq 1-\eps^p/K$$
Ribe's theorem implies that these properties are really metric properties, suggesting that each should have a reformulation that involves only the metric structure of the Banach space and not the linear structure. The research program concerned with finding these reformulations is known as the \emph{Ribe program}. The program was initiated by Bourgain in \cite{Bourgain} in which he made the first substantial contribution by characterizing superreflexive Banach spaces as those which do not admit biLipschitz embeddings of the binary trees of depth $k$ with uniform control on the biLipschitz distortion. We record here that the \emph{biLipschitz distortion} (or just distortion) of a map $f: X \to Y$ between metric spaces $(X,d_X)$, $(Y,d_Y)$ is the least value of $L$ for which there exists $0<D<\infty$ so that
$$d_X(x,y) \leq Dd_Y(f(x),f(y)) \leq Ld_X(x,y)$$
for all $x,y \in X$, that $f$ is a \emph{biLipschitz embedding} if its distortion is finite, and that $f$ is a \emph{biLipschitz equivalence} if it is a biLipschitz embedding and surjective. The \emph{biLipschitz distortion of $X$ into $Y$} is the infimal distortion of all maps from $X$ into $Y$. Another major contribution to the Ribe program is a purely metric reformulation of $p$-convexity. The metric property \emph{Markov $p$-convexity} was originally defined by Lee-Naor-Peres in \cite{LNP} and proved by Mendel-Naor in \cite{MN} to be a reformulation of $p$-convexity. Here are the specifics:
 
\begin{definition}[Definition 1.2, \cite{MN}] \label{def:Markovdef}
Let $\{X_t\}_{t \in \Z}$ be a Markov chain on a state space $\Omega$. Given an integer $k \geq 0$, we denote by $\{\tilde{X}_t(k)\}_{t \in \Z}$ the process which equals $X_t$ for time $t \leq k$ and evolves independently (with respect to the same transition probabilities) for time $t > k$. Fix $p > 0$. A metric space $(M,d)$ is called \emph{Markov $p$-convex} if there is $\Pi < \infty$ so that for every Markov chain $\{X_t\}_{t \in \Z}$ on a state space $\Omega$, and for every $f: \Omega \to M$,
$$\sum_{k=0}^{\infty} \sum_{t \in \Z} \frac{\E[d(f(X_t),f(\tilde{X}_t(t-2^k)))^p]}{2^{kp}} \leq \Pi^p \sum_{t \in \Z} \E[d(f(X_{t+1}),f(X_t))^p]$$
Set $\Pi_p(M)$ equal to the least value of $\Pi$ so that the above inequality holds (whenever it exists). $\Pi_p(M)$ is called the \emph{Markov $p$-convexity constant} of $M$.
\end{definition}

\begin{theorem}[Theorem 1.3, \cite{MN}] \label{thm:MNmain}
A Banach space is $p$-convex if and only if it is Markov $p$-convex.
\end{theorem}

Observe the following fact: if there is a map $f: X \to Y$ with biLipschitz distortion $L$, then $\Pi_p(X) \leq L\Pi_p(Y)$. Thus, Markov convexity can be used to answer quantitative questions about metric spaces in the Lipschitz category.

We present two such applications, the first on the impossibility of dimension reduction in trace class operators, $S_1$. From page 2 of \cite{NPS}): A Banach space $(X,\|\cdot\|_X)$ admits \emph{metric dimension reduction} if there exists $\alpha < \infty$ such that every $n$-point subset of $X$ biLipschitz embeds with distortion $\alpha$ into a linear subspace of $X$ with dimension $n^{o(1)}$. This definition is inspired by the famous Johnson-Lindenstrauss Lemma (\cite{JL}) which implies Hilbert space admits metric dimension reduction. In \cite{NPS}, Naor, Pisier, and Schechtman showed that there is an infinite sequence of $n$-point subsets of $S_1$ whose Markov 2-convexity constant is bounded below by a universal constant times $\sqrt{\ln(n)}$, and that the Markov 2-convexity constant of any $d$-dimensional linear subspace of $S_1$ is bounded above by a universal constant times $\sqrt{\ln(n)}$. Together these imply their main result (Theorem 1, \cite{NPS}): $S_1$ does not admit dimension reduction. For more on the Ribe program and dimension reduction, see the surveys \cite{NaorRibe} and \cite{NaorICM}.

Here is a second application of Markov convexity. In the spirit of the Ribe program, Ostrovskii found a purely metric characterization of the Radon-Nikodym property (RNP) of Banach spaces by showing that a Banach space has the RNP if and only if it does not contain a biLipschitz copy of a thick family of geodesics (Corollary 1.5 \cite{OsRNP}). He asked a natural follow-up question: if a geodesic metric space does not biLipschitz embed into any RNP space, must it contain a biLipschitz copy of a thick family of geodesics? The Heisenberg group is a geodesic metric space that does not biLipschitz embed into any RNP space (see Section 1.2 of \cite{LN} or Theorem 6.1 of \cite{CK}), and Ostrovskii showed that in fact it does not contain a biLipschitz copy of a thick family of geodesics, thus negatively answering the question. He accomplished this by proving that any metric space containing a biLipschitz copy of a thick family of geodesics cannot be Markov $p$-convex for any $p > 0$ (Theorem 1.5, \cite{OsMarkov}), and applying either of the following results of Li:

\begin{theorem} \label{thm:Li}
~\\
Proposition 7.2 and Theorem 7.4, \cite{LiCoarse}: Every graded nilpotent Lie group of step $r$ is Markov $2(r!)^2$-convex. \\
Theorem 1.1 and Corollary 1.3, \cite{LiMarkov}: The set of $p$ for which the Heisenberg group is Markov $p$-convex is exactly $[4,\infty)$.
\end{theorem}

\subsection{Summary of Results} \label{ss:results}

This article continues the line of research started by Theorem \ref{thm:Li}. Our main results are:

\begin{reptheorem}{thm:upperboundmain}
Every graded nilpotent Lie group of step $r$, equipped with a left invariant metric homogeneous with respect to the dilations induced by the grading, is Markov $p$-convex for every $p \in [2r,\infty)$.
\end{reptheorem}

\begin{reptheorem}{thm:lowerboundmain}
For every $p > 0$, $r \geq 1$, coarsely dense set $N \sbs J^{r-1}(\R)$, and $R \geq 3$, let $B_N(R) := \{x \in N: d_{CC}(0,x) \leq R\}$. Then
$$\Pi_p(B_N(R)) \gtrsim \frac{\ln(R)^{\frac{1}{p}-\frac{1}{2r}}}{\ln(\ln(R))^{\frac{1}{p}+\frac{1}{2r}}}$$
where the implicit constant can depend on $r,p$ but not on $N,R$.
\end{reptheorem}

Recall that a subset $N$ of a metric space $(X,d_X)$ is \emph{coarsely dense} if there exists $C < \infty$ such that $X = \cup_{x' \in N} \{x \in X: d_X(x,x') \leq C\}$. See Section \ref{sec:prelim} for the definition of $J^{r-1}(\R)$. Theorem \ref{thm:upperboundmain} is restated and proved at the end of Section \ref{ss:upperbound}, and similarly for Theorem \ref{thm:lowerboundmain} at the end of Section \ref{ss:lowerbound}.

We can extend this result to other groups using the notion of subquotients. Recall that a surjective map $f: X \to Y$ between metric spaces $(X,d_X),(Y,d_Y)$ is a \emph{Lipschitz quotient map} with constant $C < \infty$ if there exists $0<D<\infty$ such that for all $x \in X$ and $R > 0$,
$$B_{R}(f(x)) \sbs f(B_{DR}(x)) \sbs B_{CR}(f(x))$$
If such a map $f$ exists we say $Y$ is a Lipschitz quotient of $X$. $X$ is a \emph{Lipschitz subquotient} of $Y$ with constant $C$ if there is a metric space $Z$ such that $Z$ embeds isometrically into $Y$ and $X$ is a Lipschitz quotient of $Z$ with constant $C$, or, equivalently, there is a a metric space $Z$ such that $Z$ is a Lipschitz quotient of $Y$ with constant $C$ and $X$ isometrically embeds into $Z$. It follows from Proposition 4.1 of \cite{MN} that if $X$ is a Lipschitz subquotient of $Y$ with constant $C$ then $\Pi_p(X) \leq C \Pi_p(Y)$.

Every free Carnot group of step $r \geq 2$ has $J^{r-1}(\R)$ (in fact every graded nilpotent Lie group of step $r$ with 2-dimensional horizontal layer) as a graded quotient group, and the projection map $\R^k \twoheadrightarrow \R$ dualizes to a graded embedding $J^{r-1}(\R) \hookrightarrow J^{r-1}(\R^k)$. See Chapter 14 of \cite{BLU} for background on free Carnot groups and \cite{War} for background on the jet spaces groups $J^{r-1}(\R^k)$.

\begin{corollary} \label{cor:main}
Let $G$ be a Carnot group of step $r$ that has $J^{r-1}(\R)$ as a graded subquotient group, for example $G$ may be a free Carnot group, $J^{r-1}(\R^k)$, or any Carnot group if $r \leq 3$. The set of $p > 0$ for which $G$ is Markov $p$-convex is exactly $[2r,\infty)$.
\end{corollary}

\begin{proof}
This follows from Theorems \ref{thm:upperboundmain} and \ref{thm:lowerboundmain} and the preceding discussion.
\end{proof}

Recall that a subgroup $\Gamma \leq G$ of a Lie group $G$ is a \emph{lattice} if the subspace topology on $\Gamma$ is discrete and $G/\Gamma$ carries a $G$-invariant, Borel probability measure. 

\begin{corollary}
Let $G$ be a Carnot group of step $r$ that has $J^{r-1}(\R)$ as a graded subquotient group, for example $G$ may be a free Carnot group, $J^{r-1}(\R^k)$, or any Carnot group if $r \leq 3$ (by Lemma \ref{lem:EngelSubQuot}). Let $\Gamma \leq G$ be a lattice equipped with the word metric with respect to a finite generating set (which exists by Theorem 2.21 of \cite{Rag}), and let $B_\Gamma(R)$ denote the ball of radius $R$ in $\Gamma$ centered at the identity. Then for any $p > 0$,
$$\Pi_p(B_\Gamma(R)) \gtrsim \frac{\ln(R)^{\frac{1}{p}-\frac{1}{2r}}}{\ln(\ln(R))^{\frac{1}{p}+\frac{1}{2r}}}$$
\end{corollary}

\begin{proof}
Let $G,\Gamma,p$ be as above. The inclusion $\Gamma \hookrightarrow G$ is a biLipschitz embedding onto a coarsely dense subset when $\Gamma$ is equipped with the word metric with respect to a finite generating set (this can be proven using Mostow's theorem that lattices in nilpotent Lie groups are cocompact (\cite{Mostow}) and applying the fundamental theorem of geometric group theory). Thus it suffices to prove the conclusion for any coarsely dense $N'' \sbs G$. Let $N''$ be such a subset. By assumption, there is a Carnot group $G'$ and a graded quotient homomorphism $q: G \to G'$ such that  $J^{r-1}(\R)$ is a graded subgroup of $G'$. Then $q$ is a Lipschitz quotient map, so there is a constant $C < \infty$ such that for any $R \geq 3$, 
$$\Pi_p(B_{N''}(R)) \gtrsim \Pi_p(B_{q(N'')}(R/C))$$
Thus it suffices to prove the conclusion for any coarsely dense subset $N' \sbs G'$. Let $N'$ be such a subset. Fix $B >> 1$ and let $N \sbs J^{r-1}(\R)$ be a coarsely dense, $B$-separated subset (each pair of distinct points in $N$ is separated by a distance at least $B$ - such sets always exist by Zorn's Lemma). Then since $J^{r-1}(\R)$ is a graded subgroup of $G'$, there is a biLipschitz embedding $N \to G'$. If $B$ is chosen large enough, we map postcompose with a nearest neighbor map $G' \to N'$ to obtain another biLipschitz embedding $N \to N'$. Then the conclusion follows from Theorem \ref{thm:lowerboundmain}.
\end{proof}

The following quantitative nonembeddability estimate follows from the previous corollary and Theorem \ref{thm:upperboundmain}.

\begin{corollary}
Let $G$ be a Carnot group of step $r$ that has $J^{r-1}(\R)$ as a graded subquotient group, for example $G$ may be a free Carnot group, $J^{r-1}(\R^k)$, or any Carnot group if $r \leq 3$. Let $\Gamma \leq G$ be a lattice equipped with the word metric with respect to a finite generating set, and let $B_\Gamma(R)$ denote the ball of radius $R$ in $\Gamma$ centered at the identity. Let $G'$ be any graded nilpotent Lie group of step $r' < r$. Then we have the following estimate for $c_{G'}(B_\Gamma(R))$, the biLipschitz distortion of $B_{\Gamma}(R)$ in $G'$:
$$c_{G'}(B_\Gamma(R)) \gtrsim \frac{\ln(R)^{\frac{1}{2r'}-\frac{1}{2r}}}{\ln(\ln(R))^{\frac{1}{2r'}+\frac{1}{2r}}}$$
where the implicit constant depends on $G$ and $G'$ but not on $R$.
\end{corollary}

Such quantitative nonembeddability estimates have been the subject of much attention for embeddings of Heisenberg groups into certain Banach spaces, see \cite{ANT} and \cite{LafNaor} for uniformly convex Banach space targets and \cite{NY} for $L^1$ targets. In particular, it can be deduced from \cite{ANT} and \cite{Assouad} that the biLipschitz distortion of the ball of radius $R$ in a lattice in the Heisenberg group into Hilbert space equals, up to universal factors, $\sqrt{\ln(R)}$. Thus, our estimates in the previous corollary cannot be sharp when $r=2$ and $r'=1$. However, these estimates seem to be the first of their type when the target is allowed to be a nilpotent group of step larger than 1. Other quantitative nonembeddability estimates of between Carnot groups were obtained in \cite{LiCoarse}, but they are of a different flavor. Since our estimates are not sharp for $r=2, r'=1$, we speculate that they are not sharp for larger values of $r,r'$ either.

Next, we obtain new results on the nonexistence Lipschitz subquotient maps.

\begin{corollary} \label{cor:finiteCarnot}
Let $G$ be a Carnot group of step $r$ that has $J^{r-1}(\R)$ as a graded subquotient group, for example $G$ may be a free Carnot group, $J^{r-1}(\R^k)$, or any Carnot group if $r \leq 3$. Let $G'$ be any graded nilpotent Lie group of step $r'$.
\begin{enumerate}
\item \label{finiteCarnot1} $G$ is not a Lipschitz subquotient of $L^p$ (or any $p$-convex space) for any $p \in (1,2r)$.
\item \label{finiteCarnot2} If $r > r'$, $G$ is not a Lipschitz subquotient of $G'$.
\end{enumerate}
\end{corollary}

\begin{proof}
These follow from the previous corollary, the fact that Markov $p$-convexity is preserved under Lipschitz subquotients, Theorem \ref{thm:MNmain}, and the classical fact that $L^p$ is $\max(2,p)$-convex for $p > 1$. 
\end{proof}

Essentially all of the previously know results of this flavor are proved via Pansu differentiation (\cite{Pansu}), which applies when the domain is a (finite dimensional) Carnot group and the target is an RNP Banach space or (finite dimensional) Carnot group (Section 1.2 of \cite{LN} or Theorem 6.1 of \cite{CK}, which are stated for biLipschitz maps on the Heisenberg group, but also apply to biLipschitz or Lipschitz quotient maps on any Carnot group of step at least 2). There is also a recent differentiation theorem of Le Donne-Li-Moisala (\cite{LLM}) which applies when the domain is a ``scalable" group filtrated by (finite dimensional) Carnot groups and the target is an RNP space. However, there does not seem to be a clear way to deduce Corollary \ref{cor:finiteCarnot} in full generality from any of these methods.

We may use Markov convexity again to prove nonexistence of subquotient maps onto some ``infinite step" graded Lie groups. See Section \ref{ss:InfStep} for the definitions of inverse limits, $J^\infty(\R^k)$, and the free Carnot group on $k$ generators, $F_k^\infty$.

\begin{corollary} \label{cor:infinitecarnot}
Let $G_0 \leftarrow G_1 \leftarrow \dots$ be an inverse system of graded nilpotent Lie groups such that for every $r$, there is an $i$ with $J^{r-1}(\R)$ a graded subquotient of $G_i$, and let $G_\infty$ be the inverse limit group. For example, $G_\infty$ may be $J^\infty(\R^k)$ or $F_k^\infty$. Then $G_\infty$ is not a Lipschitz subquotient of any superreflexive space.
\end{corollary}

\begin{proof}
Pisier's renorming theorem, Theorem 11.37 of \cite{Pi}, states that any superreflexive Banach space is $p$-convex for some $p \in [2,\infty)$. Thus it suffices to show that $G_\infty$ is not Markov $p$-convex for any $p \in (0,\infty)$. For every $r \geq 1$, $J^{r-1}(\R)$ is a Lipschitz subquotient of $G_\infty$, so since Markov $p$-convexity is preserved under Lipschitz quotients, the conclusion follows from Corollary \ref{cor:main}.
\end{proof}

Finally, we provide a positive result on the existence of embeddings using one of the main results of \cite{LNP}. A \emph{metric tree} is the vertex set of a weighted graph-theoretical tree equipped with the shortest path metric.

\begin{theorem}[Theorem 4.1, \cite{LNP}] \label{thm:LNPtree}
If $T$ is a metric tree and $T$ is Markov $p$-convex, then $T$ biLipschitz embeds into $L^p$.
\end{theorem}

\begin{corollary}
If a metric tree $T$ is a Lipschitz subquotient of a graded nilpotent Lie group $G$ of step $r$, then $T$ biLipschitz embeds into $L^p$ for every $p \geq 2r$.
\end{corollary}

\begin{proof}
This follows from Theorem \ref{thm:upperboundmain}, the fact that Markov convexity is inherited by Lipschitz subquotients, and Theorem \ref{thm:LNPtree}.
\end{proof}

We conclude this introduction with the obvious conjecture that Theorems \ref{thm:upperboundmain} and \ref{thm:lowerboundmain} lead to, and another somewhat less obvious conjecture.

\begin{conjecture}
Every Carnot group of step $r$ is not Markov $p$-convex for every $p \in (0,2r)$.
\end{conjecture}

\begin{conjecture}
For each graded nilpotent Lie group $G$, the set of $p$ for which $G$ is Markov $p$-convex is the same as that of the largest Carnot subgroup of $G$.
\end{conjecture}

\section{Discussion of Proof Methods}
We engage here in informal discussion of the proofs of Theorem \ref{thm:upperboundmain} and \ref{thm:lowerboundmain}. This discussion is intended to give a brief overview of the proofs for readers with a sufficient background in the relevant topics. For Theorem \ref{thm:upperboundmain}, the relevant topics are graded nilpotent Lie algebras, the group structure they inherit via the Baker-Campbell Hausdorff formula, and their graded-homogeneous group quasi-norms. For Theorem \ref{thm:lowerboundmain}, the relevant topics are Markov convexity of diamond-type graphs, jet space Carnot groups, and Khintchine's inequality. Readers unfamiliar with these topic may find this section unuseful.

\subsection{Discussion of Proof of Theorem \ref{thm:upperboundmain}}
The method employed by Mendel-Naor to prove that $p$-convexity of Banach spaces implies Markov $p$-convexity is to:
\begin{enumerate}
\item Invoke the well-known result that $p$-convex Banach spaces have equivalent norms $\|\cdot\|$ satisfying the parallelogram inequality $(\|x\|^p + \|x-y\|^p)/2 - \|y/2\|^p \gtrsim \|x-y/2\|^p$.
\item Prove the 4-point inequality $(2d(y,x)^p+d(z,y)^p+d(y,w)^p)/2 - (d(x,w)/2)^p - (d(x,z)/2)^p \gtrsim d(z,w)^p$, where $d(x,y) = \|x-y\|$.
\item Prove the Markov $p$-convexity inequality, Definition \ref{def:Markovdef}.
\end{enumerate}

We prove the analogous inequalities for graded nilpotent Lie groups:
\begin{enumerate}
\item \label{1} Lemma \ref{lem:convnorm2}. Construct a group quasi-norm $N$ satisfying $(N(x)^p + N(y^{-1}x)^p)/2 - (N(y)/2)^p \gtrsim N(\delta_{1/2}(y)^{-1}x)^p$.
\item \label{2} Lemma \ref{lem:convmetric}. Prove the 4-point inequality $(2d(y,x)^p+d(z,y)^p+d(y,w)^p)/2 - (d(x,w)/2)^p - (d(x,z)/2)^p \gtrsim d(z,w)^p$, where $d(x,y) = N(y^{-1}x)$.
\item \label{3} Prove Theorem \ref{thm:upperboundmain}. The Markov $p$-convexity inequality.
\end{enumerate}

The passage from (\ref{1}) to (\ref{2}) and from (\ref{2}) to (\ref{3}) is exactly the same as in Banach space case. To prove (\ref{1}), we recursively construct a sequence of homogeneous quasi-norms on the group, and prove that they satisfy (1) inductively. Actually, the following stronger version of (\ref{1}) (with $p=2s$, the case $p \geq 2s$ is taken care of later) is needed for the induction to close, this is Lemma \ref{lem:convnorm1}.
$$(N_s(x)^{2s}+N_s(y^{-1}x)^{2s})/2 - (N_s(y)/2)^{2s} \gtrsim SN_s(x,y)^{2s} + D_s(x,y) + N_s(\delta_{1/2}(y)^{-1}x)^{2s}$$
There are two extra terms that appear in this inequality, $SN_s(x,y)$ and $D_s(x,y)$, defined in Definitions \ref{def:Ddef} and \ref{def:SNdef}. $D_s(x,y)$ is designed to bound (up to constants) the square of any BCH polynomial of degree $s$ (see Definition \ref{def:BCHdef}), so one may guess how it would be useful to prove (\ref{1}).

$SN_s(x,y)$ is nearly a positive definite quasi-norm of $(x_1,\dots x_s,y_1, \dots y_s)$ (the name $SN$ is meant to suggest that it is a seminorm instead of a norm, since it is not positive definite), but not quite as it vanishes when $x_1 = y_1/2$ and $x_i=y_i=0$ for $i \geq 2$. However, this is not an issue as we will have an extra $\|y_1\|$ term in the induction, so that $\|y_1\| + SN_s(x,y)$ is genuinely a quasi-norm of $(x_1,\dots x_s,y_1, \dots y_s)$. Here are $D_s$ and $SN_s$ for some small $s$:
$$D_3(x,y) = \|(x_3,y_3)\|^2 + \|(x_1,y_1)\|^2\|(x_2,y_2)\|^2 + \|(x_1,y_1)\|^2\boldsymbol{\tau}^2(x,y)$$
$$D_4(x,y) = \|(x_4,y_4)\|^2 + \|(x_1,y_1)\|^2\|(x_3,y_3)\|^2 + \|(x_2,y_2)\|^4$$
$$+ \|(x_1,y_1)\|^4\|(x_2,y_2)\|^2 + \|(x_2,y_2)\|^2\boldsymbol{\tau}^2(x,y) + \|(x_1,y_1)\|^4\boldsymbol{\tau}^2(x,y)$$
$$SN_3(x,y) = \max(\|x_1-y_1/2\|,\|(x_2,y_2)\|^{1/2},\|(x_3,y_3)\|^{1/3})$$
The polynomial $\boldsymbol{\tau}^2(x,y)$ is designed to bound the squares of terms coming from the bracket between two vectors from the horizontal layer. For example, in the second Heisenberg group,
$$\boldsymbol{\tau}^2(x,y) = (x_{11}y_{12}-x_{12}y_{11})^2 + (x_{13}y_{14}-x_{14}y_{13})^2$$
We recursively construct the quasi-norms $N_{s+1}$ given all the previous quasi-norms by defining $N_{s+1}(x)$ to be an $\ell^{2(s+1)}$ sum of $\lambda_{s+1}\|x_{s+1}\|^{1/(s+1)}$ and the top half of the previously defined quasi-norms, where $\lambda_{s+1}$ is a positive constant chosen small enough (depending on the product structure of the group in question) to make the inequality of Lemma \ref{lem:convnorm1}(\ref{convnorm1}) hold. Specifically, from \eqref{eq:Ndef},
$$N_2(x) = \sqrt[4]{\|x_1\|^4+\lambda_2\|x_2\|^2}$$
$$N_{s+1}(x) = \sqrt[2(s+1)]{\lambda_{s+1}\|x_{s+1}\|^2 + \sum_{s'= \lceil (s+1)/2 \rceil}^s N_{s'}^{2(s+1)}(x)}$$
The reason why we add the top half of the previously defined norms, and the reason for the inclusion $SN_s(x,y)$ term in the inequality, is to help pass from $D_s(x,y)$ to $D_{s+1}(x,y)$ during the proof of the inductive step. When proving the inductive step, we have terms like \\ $(SN_{s'}(x,y)^{2s'} + D_{s'}(x,y))^{(s+1)/s'}$, $s' \leq s$, appearing to which we apply Lemma \ref{lem:binomineq} and obtain a term like $SN_{s'}(x,y)^{2(s+1-s')}D_{s'}(x,y)$. This term bounds $\|(x_{s+1-s'},y_{s+1-s'})\|^2D_{s'}(x,y)$ exactly when $\lceil (s+1)/2 \rceil \leq s' \leq s$. Then summing $\|(x_{s+1-s'},y_{s+1-s'})\|^2D_{s'}(x,y)$ over this range of $s'$ accounts for all the terms in $D_{s+1}(x,y)$, except for the top-layer term $\|(x_{s+1},y_{s+1})\|^2$ (since any other term in $D_{s+1}(x,y)$ contains as a factor a variable from one of the lower half layers, see Lemma \ref{lem:domBCH1} for details), which is accounted for later.

\subsection{Discussion of Proof of Theorem \ref{thm:lowerboundmain}}
We recursively construct a sequence of directed graphs $\Gamma_m$ and maps from them into the jet space of step $r$ ($J^{r-1}(\R)$) to show that it is not Markov $p$-convex for any $p < 2r$. The Markov processes we use are standard directed random walks on the graphs. This is very similar to the method used in \cite{LiMarkov}, where something akin to the Laakso-Lang-Plaut diamond graphs were used. The main feature of those graphs $G_m$ is that $G_{m+1}$ is obtained from $G_i$ by replaced each edge of $G_1$ with a copy of $G_m$. Roughly speaking, Li recursively maps $G_{m+1}$ into $\R^2$ by replacing each edge of a distorted image of $G_1$ by a rotated, distorted copy of the image of $G_i$. The distortion is done in such a way that the coLipschitz constant (the Lipschitz constant of the inverse map) is on the order of $\sqrt[4]{m}\sqrt{\ln(m+1)}$, and the fact that rotations are isometries of the Heisenberg group affords one uniform control on the Lipschitz constants. One can conclude from this that the Heisenberg group is not Markov $p$-convex for $p < 4$ (the 4 coming from the fourth root of $m$).

Our graphs differ from those in \cite{LiMarkov} in that, to obtain $\Gamma_{m+1}$ from $\Gamma_m$, we first glue together \emph{many} copies of $\Gamma_{m}$ together with a small number of copies of a single edge $I$ in series to get a new graph $\Gamma_{m+1}'$, and then replace each edge of $\Gamma_1$ with a copy of $\Gamma_{m+1}'$ (this isn't exactly how our construction is defined, but is close enough to get the main idea). See Definition \ref{def:graphs} for the full details. We will explain the reasoning for this after describing our maps of $\Gamma_m$ into $J^{r-1}(\R)$.

Our maps differ from those in \cite{LiMarkov} in that we \emph{do not} rotate the image of $\Gamma_m$ before using it to replace the edges of the image of $\Gamma_1$, as rotations are not Lipschitz maps in higher step groups like they are in the Heisenberg group. Refer to Figure \ref{fig:Fm} throughout this discussion to get an idea of the construction of these maps. Instead of rotating, we simply add (many copies of) the image of $\Gamma_m$ to a distorted copy of the image of $\Gamma_1$ to obtain the mapping of $\Gamma_{m+1}$ into $\R^2$. More specifically, we map each directed path $\gamma$ in $\Gamma_{m+1}$ to the jet of a function $\phi_\gamma$ - a horizontal curve in $J^{r-1}(\R)$. The Lipschitz constant of this map is controlled by $\left\|\frac{d^r}{d^rx}\phi_\gamma\right\|_\infty$. We still distort the graphs $\Gamma_m$ with the same asymptotics as in \cite{LiMarkov}, so that the coLipschitz constant is on the order of $\sqrt[2r]{m}\sqrt[r]{\ln(m+1)}$ (at least on the pairs of random walks $(X^m_t,\tilde{X}^m_t(t-2^k)$). That we get the $2r^{\text{th}}$ root of $m$ instead of the fourth root of $m$ comes from the fact that $J^{r-1}(\R)$ is of step $r$ and the Heisenberg group is of step 2. One potential problem is that the absence of isometric rotations and the fact that $(\sqrt{m}\ln(m))^{-1}$ isn't summable means $\left\|\frac{d^r}{d^rx}\phi_\gamma\right\|_\infty$ blows up along some paths, and thus we do not have uniform control on the Lipschitz constant of the map, unlike \cite{LiMarkov}. However, $(\sqrt{m}\ln(m))^{-1}$ \emph{is square-summable}, and together with the nature of the image of the random walk $X^m_t$ in $J^{r-1}(\R)$, this allows us to control $\E[d_{CC}(X^m_{t+1},X^m_t)^p]$ uniformly in $m,t$. Loosely, along the random walk in the horizontal layer (which has $x$- and $u_{r-1}$-coordinates), every time one is confronted with a choice of direction to walk in, the choice is to walk 1 unit in the $x$-direction and $+(\sqrt{i}\ln(i+1))^{-1}$ units in the $u_{r-1}$-direction with probability 1/2, or 1 unit in the $x$-direction and $-(\sqrt{i}\ln(i+1))^{-1}$ units in the $u_{r-1}$-direction with probability 1/2 (for some $i$ depending on how far one has walked). Thus, one might expect $d_{CC}(X^m_{t+1},X^m_t)$ to be bounded by a random variable distributed like $1+|\sum_{i=1}^t \eps_i(\sqrt{i}\ln(i+1))^{-1}|$, where $\{\eps_i\}_i$ are iid Rademachers, and then Khintchine's inequality implies we should have a uniform bound on $\E[d_{CC}(X^m_{t+1},X^m_t)^p]$ (which is the real quantity of interest, recall Definition \ref{def:Markovdef}). Of course, the random walk is not distributed like this, but it turns out that this intuition is correct nonetheless, see Lemmas \ref{lem:subgaussian} and \ref{lem:mapintojetspace}(\ref{mapintojetspace4}) for the specifics.

Finally, the reason we use many copies of $\Gamma_m$ in creating $\Gamma_{m+1}$ is so that, compared to the diameter of $\Gamma_{m+1}$, the diameter of the copies of $\Gamma_m$ is very small, and thus those that replaced opposite edges of $\Gamma_1$ don't get too close together, which would ruin the coLipschitz constant. Morally, this ``decouples" any interaction between different scales in $\Gamma_{m+1}$.

\section{Preliminaries} \label{sec:prelim}
The next two subsections don't follow any particular reference, but ones we recommend are \cite{BLU} for Carnot groups and \cite{LeDonne} for graded nilpotent groups. We mostly follow \cite{War} for the subsection on jet spaces.
\subsection{Graded Nilpotent and Stratified Lie Algebras and their Lie Groups} \label{ss:CarnotGroups}
A \emph{graded nilpotent Lie algebra $(\g,[\cdot,\cdot])$ of step $r$} is a Lie algebra equipped with a grading $\g = \oplus_{i=1}^r \g_i$, meaning $\g_r \neq 0$, $[\g_i,\g_j] \sbs \g_{i+j}$ if $i+j \leq r$, and $[\g_i,\g_j] = 0$ if $i+j > r$. A \emph{stratified} Lie algebra $(\g,[\cdot,\cdot])$ of \emph{step} $r$ is a graded nilpotent Lie algebra of step $r$ such that the Lie subalgebra generated by $\g_1$ is all of $\g$. The grading is called a \emph{stratification}, $\g_1$ is often called the \emph{horizontal layer} (or stratum), and $\g$ is said to be \emph{horizontally generated}. Whenever a Lie algebra $\g$ (not presumed to be equipped with a grading) admits a stratification, it is unique (Lemma 2.16, \cite{LeDonne}). A \emph{graded nilpotent Lie group of step $r$} is a simply connected Lie group whose Lie algebra is graded nilpotent of step $r$. A graded nilpotent Lie group whose Lie algebra is stratified is a \emph{Carnot group}. A \emph{graded homomorphism} or \emph{map} is a Lie group homomorphism between graded nilpotent Lie groups whose derivative is a graded Lie algebra homomorphism. One graded nilpotent Lie group $G'$ is a \emph{graded subgroup} of another graded nilpotent Lie group $G$ if there is an injective graded homomorphism from $G'$ into $G$. One graded nilpotent Lie group $G'$ is a \emph{graded quotient group} of another graded nilpotent Lie group $G$ if there is a surjective graded homomorphism from $G$ onto $G'$. One graded nilpotent Lie group $G'$ is a \emph{graded subquotient group} of another graded nilpotent Lie group $G$ if there is another graded nilpotent Lie group $G''$ such that $G''$ is a graded subgroup of $G$ and $G'$ is a graded quotient group of $G''$, or, equivalently, there is another graded nilpotent Lie group $G''$ such that $G''$ is a graded quotient group of $G$ and $G'$ is a graded subgroup of $G''$.

Given a graded nilpotent Lie group $G$ and its Lie algebra $\g$, since $\g$ is nilpotent and $G$ is simply connected, the exponential map is a diffeomorphism, and thus we can use it to equip $\g$ with a graded nilpotent Lie group structure such that it becomes graded isomorphic to $G$. The Baker-Campbell-Hausdorff formula provides a formula for the group product on $\g$ in terms of the Lie algebra structure (Section 2, \cite{War}):
\begin{equation} \label{eq:BCH}
xy = \sum_{n > 0} \frac{(-1)^{n+1}}{n} \sum_{\substack{0<p_i+q_i \\ i \leq i \leq n}} C_{p,q}^{-1}(\ad x)^{p_1}(\ad y)^{q_1} \dots (\ad x)^{p_n}(\ad y)^{q_n-1}y
\end{equation}
where $(\ad x)y = [x,y]$ and $C_{p,q} = p_1!q_1! \dots p_n!q_n! \left(\sum_{i=1}^n p_i+q_i\right)$. In this formula and what follows, whenever $\g$ is a graded nilpotent Lie algebra, we equip it with the product defined by \eqref{eq:BCH} and simultaneously think of $\g$ as a graded nilpotent Lie group and Lie algebra. We will always use juxtaposition to denote the group product.

Every graded nilpotent Lie group $G$ has a canonical family of \emph{dilations} $\delta_t: G \to G$ parametrized by $t \in (0,\infty)$ whose derivative $\delta_t': \g \to \g$ is defined by 
$$\delta_t'(x) := tx_1 + t^2x_2 + \dots t^rx_r$$
where $\g$ is the Lie algebra, and $x_i \in \g_i$ is the $\g_i$-component of $x \in \g$. $t \mapsto \delta_t$ is an automorphic $\R_{> 0}$-action on $G$. It can be deduced that a Lie group homomorphism $\theta$ between graded nilpotent Lie groups is a graded homomorphism if and only if it is $\delta_t$-equivariant, that is, $\theta(\delta_t(x)) = \delta_t(\theta(x))$, where we've abused (and will continue to do so) notation and written $\delta_t$ for the dilation on both the domain and codomain.

\subsection{Norms and Metrics} \label{ss:NormsMetrics}
Let $G$ be a graded nilpotent Lie group. A \emph{homogeneous quasi-norm} on $G$ is a continuous function $N: G \to \R$ such that for all $x \in G$ and $t \in \R_{> 0}$,
\begin{itemize}
\item $N(x) \geq 0$ (positive semi-definite)
\item $N(x^{-1}) = N(x)$ (symmetry)
\item $N(\delta_t(x)) = tN(x)$ (homogeneity)
\end{itemize}
If additionally $N(x) = 0$ implies $x = 0$, then $N$ is a \emph{positive definite} homogeneous quasi-norm, and if $N(xy) \leq N(x) + N(y)$ for all $x,y \in G$ (triangle inequality), $N$ is a \emph{homogeneous norm}. For any two positive definite homogeneous quasi-norms $N,N'$ on $G$, the continuity, homogeneity, and positive definiteness of $N,N'$, together with the compactness of the unit sphere in $\oplus_{i=1}^r \R^{\dim(\g_i)}$, imply that $N$ and $N'$ are biLipschitz equivalent, that is, there is a constant $0 < C < \infty$ such that
$$C^{-1}N(x) \leq N'(x) \leq CN(x)$$
for all $x \in G$.

Positive definite homogeneous norms always exist, most famously those considered in \cite{HS}. Thus any positive definite homogeneous quasi-norm $N$ satisfies the \emph{quasi-triangle inequality}: there is a $0 < C < \infty$ such that for all $x,y \in G$,
$$N(xy) \leq C(N(x) + N(y))$$

Typically one requires that every homogeneous quasi-norm $N$ satisfies the quasi-triangle inequality. Although it turns out that the quasi-norms we consider in this article do satisfy the quasi-triangle inequality, we only need to know this for positive-definite quasi-norms and thus do not explicitly make this requirement.

There is a bijective correspondence between homogeneous, positive definite quasi-norms $N$ on $G$ and left-invariant, homogeneous quasi-metrics $d_N$ on $G$ via $N \mapsto d_N$ defined by
$$d_N(x,y) := N(y^{-1}x)$$
Positive definiteness of $N$ implies positive definiteness of $d_N$, symmetry of $N$ implies symmetry of $d_N$, homogeneity of $N$ implies the homogeneity of $d_N$ (meaning $d_N(\delta_t(x),\delta_t(y)) = td_N(x,y)$), and the quasi-triangle inequality of $N$ implies the quasi-triangle inequality of $d_N$. The left-invariance of $d_N$ is automatic from the definition. $N$ satisfies the triangle inequality if and only if $d_N$ does. The inverse of $N \mapsto d_N$ is $d \mapsto N_d$, where $N_d(x) := d(0,x)$. In addition to those determined by the homogeneous, positive definite norms from \cite{HS}, there are canonical left-invariant, homogeneous metrics on Carnots groups called \emph{Carnot-Caratheodory metrics}, denoted $d_{CC}$. These metrics are also geodesic. See \cite{BLU} or \cite{LeDonne} for further information.

In what follows, whenever dealing with a graded nilpotent Lie group, we will automatically assume it is equipped with a left-invariant, homogeneous quasi-metric. By the preceding discussion, this quasi-metric is well-defined up to biLipschitz equivalence, so any biLipschitz-invariant property of metric spaces we may well attribute to a graded nilpotent Lie group $G$ knowing only the algebraic structure of its graded Lie algebra. The $\delta_t$-equivariance of graded group maps implies that any graded map between graded nilpotent Lie groups is Lipschitz, and thus graded group embeddings are biLipschitz embeddings, graded quotient maps are Lipschitz quotient maps, and graded group isomorphisms are biLipschitz equivalences.

\subsection{Model Filiform Groups and Jet Spaces over $\R$}
\label{ss:MfJs}
We follow \cite{War} (especially Example 4.3) throughout this subsection. The \emph{model filiform group} of step $r \geq 1$ is the Carnot group with stratified Lie algebra $\g = (\R X \oplus \R Y_1) \oplus_{i=2}^{r} \R Y_i$, where $X,Y_1$ is a basis for $\g_1$ and $Y_i$ is a basis for $\g_i$ for $2 \leq i \leq r$, and the nontrivial bracket relations are given by $[X,Y_i] = Y_{i+1}$ for $1 \leq i \leq r-1$. Clearly, for $s \geq r$, there is a canonical Carnot group quotient map from the model filiform group of step $s$ to that of step $r$. The model filiform group of step 2 is frequently called the \emph{Heisenberg group}, and the one of step 3 the \emph{Engel group}. The corresponding Lie algebras are the \emph{Heisenberg algebra} and \emph{Engel algebra}.

The \emph{jet space over $\R$} of step $r \geq 0$, denoted $J^{r-1}(\R)$, is a certain Carnot group of step $r$ graded isomorphic to the model filiform group of step $r$. There are also jet space groups $J^{r-1}(\R^k)$ over higher dimensional Euclidean space, but we will focus on $k=1$ in this discussion. As a set, $J^{r-1}(\R)$ consists of equivalence classes of pairs $(x,f)$ where $x \in \R$ and $f \in C^{r-1}(\R)$. Two pairs $(x,f),(y,g)$ are equivalent if $x = y$ and $f^{(k)}(x) = g^{(k)}(y)$ for all $0 \leq k \leq r-1$. We define maps $\pi_x,\pi_i: J^{r-1}(\R) \to \R$, $0 \leq i \leq r-1$, by $\pi_x([(y,g)]) = y$ and $\pi_i([(y,g)]) = g^{(i)}(y)$. These maps are obviously well-defined and the direct sum map $\pi_x \oplus_{i=0}^{r-1} \pi_{r-1-i}: J^{r-1}(\R) \to \R \cross \R^{r}$ is a bijection. For $v \in J^{r-1}(\R)$, the quantity $\pi_x(v)$ is referred to as the \emph{$x$-coordinate} and $\pi_i(v)$ as the \emph{$u_i$-coordinate}. We equip $J^{r-1}(\R)$ with a topological vector space structure so that this map is a linear homeomorphism, and from this point on will represent elements of $J^{r-1}(\R)$ using these coordinates. We will especially represent elements as pairs $(y,v) \in J^{r-1}(\R) = \R \cross \R^{r}$ so that $y \in \R$, $v \in \R^{r}$, and $\pi_x((y,v)) = y$. Although we won't explicitly use it, the group operation on $J^{r-1}(\R)$ is given by
$$\pi_x((x,u_{r-1},\dots u_0)*(y,v_{r-1},\dots v_0)) = x+y$$
$$\pi_i((x,u_{r-1},\dots u_0)*(y,v_{r-1},\dots v_0)) = u_i+v_i + \sum_{j=i+1}^{r-1} u_j\frac{y^{j-i}}{(j-i)!}$$
Given $y \in \R$ and $g \in C^{r-1}(\R)$, we get an element $[j^{r-1}(y)](g) \in J^{r-1}(\R)$ defined by
$$\pi_x([j^{r-1}(y)](g)) = y$$
$$\pi_i([j^{r-1}(y)](g)) = g^{(i)}(y)$$
called the \emph{jet} of $g$ at $y$. The following two Lemmas are essentially all we need to know about jet spaces. The first is a special case of \cite{RW}. Although their lemma is stated for $C^r$ functions, the proof works the same in the case of $C^{r-1,1}$ functions.

\begin{lemma}[pages 4-5, \cite{RW}] \label{lem:dccupperbound}
For any $[a,b] \sbs \R$ and $\phi \in C^{r-1,1}([a,b])$,
$$d_{CC}([j^{r-1}(b)](\phi),[j^{r-1}(a)](\phi)) \leq \left(1 + \left\|\phi^{(r)}\right\|_{L^\infty([a,b])}\right)|b-a|$$
\end{lemma}

\begin{lemma} \label{lem:dcclowerbound}
There is a constant $c > 0$ such that for all $(x,u),(x,v) \in J^{r-1}(\R)$,
$$d_{CC}((x,u),(x,v)) \geq c|\pi_0(u-v)|^{\frac{1}{r}}$$
\end{lemma}

\begin{proof} By left invariance of $d_{CC}$ and the ball-box theorem (see Corollary 2.2 of \cite{JungSpheres}, there is a constant $c > 0$ such that for all $(x,u),(x,v) \in J^{r-1}(\R)$,
$$d_{CC}((x,u),(x,v)) \geq c|\pi_0((x,v)^{-1}(x,u))|^{\frac{1}{r}}$$
and by Lemma 3.1 from \cite{JungGromov},
$$\pi_0((x,v)^{-1}(x,u)) = \pi_0(u-v)$$
\end{proof}

The following lemma will be used to obtain lower bounds on the Markov convexity of Carnot groups of step 2 or 3.

\begin{lemma} \label{lem:EngelSubQuot}
Every Carnot group of step 2 or 3 contains the model filiform group of the corresponding step (the Heisenberg or Engel group) as a graded subquotient group.
\end{lemma}

\begin{proof}
Let $G$ be a Carnot group of step 2 with stratified Lie algebra $\g = \g_1 \oplus \g_2$. Since $\g$ has step 2, there is a nonzero $V_2 \in \g_2$. Since $\g$ is horizontally generated, there exist $U,V_1 \in \g_1$ such that $[U,V_1] = V_2$. Recall that the Heisenberg algebra has first layer generated by linearly independent vectors $X,Y_1$, second layer generated by $Y_2 \neq 0$, and nontrivial bracket relation $[X,Y_1] = Y_2$. Then it easily follows that $X \mapsto U$, $Y_1 \mapsto V_1$, $Y_2 \mapsto V_2$ is a graded algebra embedding into $\g$. This proves that the Heisenberg group is a graded subgroup of $G$.

Now assume $G$ is of step 3 with stratified Lie algebra $\g = \g_1 \oplus \g_2 \oplus \g_3$. By the grading property, any subspace of $\g_3$ is an ideal, and thus there is a graded algebra quotient map onto another step 3 stratified Lie algebra whose third layer is one dimensional. Thus we may assume $\g_3 = \R W$, $W \neq 0$, and prove that the Engel algebra embeds into $\g$. Since $\g$ is horizontally generated, $W = [U_1,[U_2,U_3]]$ for some $U_1,U_2,U_3 \in \g_1$. First we claim that there is a 2-dimensional subspace of the span of $U_1,U_2,U_3$ that generates a Lie subalgebra of step 3. After proving the claim, we'll show that this subalgebra must be graded algebra-isomorphic to the Engel algebra. To prove the claim, we'll show that at least one of the following is nonzero:
\begin{enumerate}
\item \label{liebrack1} $[U_1,[U_1,U_2]]$
\item \label{liebrack2} $[U_1,[U_1,U_3]]$
\item \label{liebrack3} $[U_2,[U_2,U_3]]$
\item \label{liebrack4} $[U_3,[U_3,U_2]]$
\item \label{liebrack5} $[U_1+U_2,[U_1+U_2,U_3]]$
\item \label{liebrack6} $[U_1+U_3,[U_1+U_3,U_2]]$
\end{enumerate}
Assume that all terms are 0. First let's see that $[U_2,[U_3,U_1]] = W$.
$$0 \overset{(\ref{liebrack5})}{=} [U_1+U_2,[U_1+U_2,U_3]] = [U_1,[U_1,U_3]] + [U_1,[U_2,U_3]] + [U_2,[U_1,U_3]] + [U_2,[U_2,U_3]]$$
$$\overset{(\ref{liebrack2}),(\ref{liebrack3})}{=} W + [U_2,[U_1,U_3]] = W - [U_2,[U_3,U_1]]$$
Using $(\ref{liebrack6}),(\ref{liebrack1}),(\ref{liebrack4})$ in place of $(\ref{liebrack5}),(\ref{liebrack2}),(\ref{liebrack3})$ shows $[U_3,[U_1,U_2]] = W$.
Putting these together yields:
$$[U_1,[U_2,U_3]] + [U_2,[U_3,U_1]] + [U_3,[U_1,U_2]] = 3W \neq 0$$
in violation of the Jacobi identity. This proves the claim.

So now the situation is that there are $Z_1,Z_2 \in \g_1$ with $[Z_1,[Z_1,Z_2]] = zW$ for some $z \neq 0$. Recall that the Engel algebra has first layer spanned by $X,Y_1$, second layer by $Y_2$, and third layer by $Y_3$ with nontrivial bracket relations $[X,Y_1] = Y_2$ and $[X,Y_2] = Y_3$. Let $z' \in \R$ such that $[Z_2,[Z_1,Z_2]] = z'W$. Then since $[Z_1,[Z_1,Z_2]] = zW \neq 0$,
the map from the Engel algebra into $\g$ defined by
$$X \mapsto Z_1, \hspace{.3in} Y_1 \mapsto Z_2 - \frac{z'}{z}Z_1, \hspace{.3in} Y_2 \mapsto [Z_1,Z_2], \hspace{.3in} Y_3 \mapsto zW$$
is a graded algebra embedding.
\end{proof}

\begin{remark}
The analogue of Lemma \ref{lem:EngelSubQuot} is false for groups of step larger than 3. Let $\g$ be the stratified Lie algebra $\g = \oplus_{i=1}^4 \g_i$ with $\g_1 = \R X_{11} \oplus \R X_{12}$, $\g_2 = \R X_2$, $\g_3 = \R X_{31} \oplus \R X_{32}$, $\g_4 = \R X_4$ and nontrivial brackets $[X_{11},X_{12}] = X_2$, $[X_{11},X_2] = X_{31}$, $[X_{12},X_2] = X_{32}$, $[X_{11},X_{31}] = X_4$, $[X_{12},X_{32}] = X_4$. The only graded quotient maps from $\g$ onto another step 4 stratified Lie algebra or graded embeddings into $\g$ from another step 4 stratified Lie algebra are isomorphisms.
\end{remark}

\subsection{Infinite Step Carnot groups}
\label{ss:InfStep}
Given an inverse system of graded nilpotent Lie groups $G_1 \overset{\rho_1}{\leftarrow} G_2 \overset{\rho_2}{\leftarrow}\dots$, where each $\rho_i$ is a graded quotient map, we define the \emph{inverse limit metric group}, $G_\infty$, to be the subgroup of $(\oplus_{i=1}^\infty G_i)_{\infty}$ consisting of those sequences $(x_i)_{i=1}^\infty$ for which $\rho(x_{i+1}) = x_i$ for all $i \geq 1$, where $(\oplus_{i=1}^\infty G_i)_{\infty}$ is the $\ell_\infty$-sum of the pointed metric spaces $(G_i,d_{CC},0)$. $G_\infty$ inherits a left-invariant homogeneous metric from $(\oplus_{i=1}^\infty G_i)_{\infty}$ (where the dilations $\delta_t$ are defined on $G_\infty$ in the obvious way), and each $G_i$ is a Lipschitz quotient of $G_\infty$.

\begin{definition}
$J^\infty(\R^k)$ is the inverse limit metric group, equipped with the induced $\delta_t$-action, associated to the natural inverse system formed by the jet space groups, $J^0(\R^k) \overset{\rho_1}{\leftarrow} J^1(\R^k) \overset{\rho_2}{\leftarrow} \dots$. See \cite{War} for background on jet space groups. Similarly, $F_k^\infty$ is the inverse limit metric group, equipped with the induced $\delta_t$-action, associated to the natural inverse system formed by the free Carnot groups on $k$ generators, $F_k^1 \overset{\rho_1}{\leftarrow} F_k^2 \overset{\rho_2}{\leftarrow} \dots$. See Chapter 14 of \cite{BLU} for background on free Carnot groups.
\end{definition}
%
%

\subsection{Probabilistic and Convexity Inequalities}
In this article, we will often justify an inequality with the phrase ``by convexity" or ``by the parallelogram law". The convexity inequalities we refer to are almost always of the form
$$\frac{a^p+b^p}{2} \geq \left(\frac{a+b}{2}\right)^p$$
or
$$a^p+b^p \leq \left(a+b\right)^p$$
for $p \geq 1$ and $a,b \geq 0$. The form of the parallelogram law we most often use is
$$\frac{\|u\|^2 + \|u-v\|^2}{2} = \|v/2\|^2 + \|u-v/2\|^2$$
for $u,v$ in a Hilbert space, which implies the inequality
$$\|u\|^2 + \|u-v\|^2\geq \frac{\|v\|^2}{2}$$
We may also use either of these inequalities without explicitly mentioning convexity or the parallelogram law.

We collect here some basic inequalities related to convexity and an additional one on $L^p$-norms of random variables.

\begin{lemma} \label{lem:binomineq}
For all $a,b \geq 0$ and $q \geq 1$,
$$(a+b)^q \geq a^q + qa^{q-1}b$$
\end{lemma}

\begin{proof}
Let $a,b,q$ be as above. The inequality is obviously true if $a = 0$. Then if $a > 0$, after dividing each side by $a^q$ and replacing $b/a$ with $t$, it suffices to prove $(1+t)^q \geq 1 + qt$. This inequality is true since the right hand is the linearization of the left hand side at $t=0$, and the left hand side is a convex function of $t$. 
\end{proof}

\begin{lemma} \label{lem:psum}
For each $p > 0$ and $k \geq 1$,
$$\sum_{t=1}^k (2t)^p > k^{p+1}/2$$
\end{lemma}

\begin{proof}
Let $p > 0$ and $k \geq 1$. Since the function $t \mapsto (2t)^p$ is increasing,
$$\sum_{t=1}^k (2t)^p > \int_{0}^k (2t)^p dt = \frac{2^p}{p+1}k^{p+1} \geq k^{p+1}/2$$
\end{proof}

The following two lemmas are frequently used in tandem to prove Khintchine's inequality (for example, Proposition 4.5 of \cite{Wolff}). We will need them for a similar inequality used in Section \ref{ss:lowerbound}.

\begin{lemma} \label{lem:coshexp}
For all $y \in \R$, $\cosh(y) \leq \exp(y^2/2)$.
\end{lemma}

\begin{proof}
Let $y \in \R$.
$$\cosh(y) = \frac{e^{y}+e^{-y}}{2} = \frac{1}{2}\sum_{k=0}^\infty \frac{y^k+(-y)^k}{k!} = \sum_{k=0}^\infty \frac{y^{2k}}{(2k)!} \leq \sum_{k=0}^\infty \frac{(y^2/2)^k}{k!} = \exp\left(\frac{y^2}{2}\right)$$
\end{proof}

\begin{lemma} \label{lem:subgaussian}
For each $p \geq 1$ and $0 < A,B < \infty$, there is a constant $C = C(p,A,B) < \infty$ such that any real-valued random variable $Y$ satisfying the moment generating function subgaussian bound
$$\E[\exp(yY)] \leq Ae^{By^2}$$
also satisfies the $L^p$-norm bound
$$\E[|Y|^p] \leq C$$
\end{lemma}

\begin{proof}
This is a standard result from the theory of subgaussian random variables whose proof appears in any text on measure concentration. For the sake of completeness we'll include the proof, roughly following the proof of Proposition 4.5 from \cite{Wolff}. Let $p$, $A$, $B$, $Y$ be as above. For any $t > 0$, Markov's inequality and our assumption imply
$$\P(Y \geq t) = \P\left(\exp\left(\frac{t}{2B}Y\right) \geq \exp\left(\frac{t^2}{2B}\right)\right) \leq \exp\left(-\frac{t^2}{2B}\right)\E\left[\exp\left(\frac{t}{2B}Y\right)\right]$$
$$\leq A\exp\left(-\frac{t^2}{2B}+\frac{t^2}{4B}\right) = A\exp\left(-\frac{t^2}{4B}\right)$$
Likewise,
$$\P(Y \leq -t) \leq A\exp\left(-\frac{t^2}{4B}\right)$$
giving us
$$\P(|Y| \geq t) \leq 2A\exp\left(-\frac{t^2}{4B}\right)$$
We then use the layer cake principle to calculate $\E[|Y|^p]$:
$$\E[|Y|^p] = p\int_0^\infty t^{p-1}\P(|Y| \geq t) dt \leq p\int_0^\infty t^{p-1}2A\exp\left(-\frac{t^2}{4B}\right) dt = C(p,A,B) < \infty$$
\end{proof}

\section{Upper Bound on Markov Convexity of Graded Nilpotent Lie Groups}
Throughout this section, fix a graded nilpotent Lie algebra $(\g,[\cdot,\cdot])$ of step $r \geq 2$ with grading $\oplus_{i=1}^r \g_i$ and dim$(\g_i) = k_i$. Choose an ordered basis $U_{i,1}, \dots U_{i,k_i}$ for each $\g_i$ and equip $\g$ with a Hilbert norm $\|\cdot\|$ such that these vectors form an orthonormal basis. We also use $\|\cdot\|$ to denote the Euclidean norm on any $\R^n$. Given $x \in \g$, let $x_i \in \g_i$ denote its $\g_i$-component. Given $x_i \in \g_i$, let $x_{i,j} \in \R$ denote its $U_{i,j}$-component. Thus,
$$\|x\|^2 = \sum_{i=1}^r \|x_i\|^2 \hspace{.5in} \text{and} \hspace{.5in} \|x_i\|^2 = \sum_{j=1}^{k_i} |x_{i,j}|^2$$
Consider $\g$ as a graded nilpotent Lie group as in Section \ref{sec:prelim}. It's easy to see that $0$ is the group identity element and $x^{-1} = -x$. Whenever $u,v \in \g$ or $u,v \in \R^n$, we use the notation $\|(u,v)\|^2$ to mean $\|u\|^2 + \|v\|^2$.

\subsection{BCH Polynomials}
\begin{definition} \label{def:BCHdef}
For $s \geq 0$, a function $P: \g \cross \g \to \R$ that is a monomial(polynomial) in the variables $x_{n,m},y_{n,m}$ is a \emph{graded-homogeneous monomial(polynomial) of degree $s$} if $P(\delta_t(x),\delta_t(y)) = t^sP(x,y)$ for all $x,y \in \g$ and $t \in \R_{>0}$. Clearly, any graded-homogeneous polynomial of degree $s$ must be a sum of graded-homogeneous monomials of degree $s$.

In this section, a \emph{multiset} is a finite sequence of positive integers modulo permutations. Disjoint unions $I_1 \sqcup I_2$ of multisets are defined in the obvious way. Given a multiset $I$, $\|I\|_1$ denotes the sum of the elements and $\|I\|_\infty$ the maximum of the elements. Given a nonzero graded-homogeneous monomial $M$ of degree $s$, we associate to it a multiset $I(M)$ defined recursively on the number of variables in the monomial by $I(M) = \{i\} \sqcup I(M')$ if $M(x,y) = x_{i,n}M'(x,y)$ or $M(x,y) = y_{i,n}M'(x,y)$ for some $n \leq k_i$ and graded-homogeneous polynomial $M'$ of degree $s-i$ (the base case is $I(1) = \emptyset$). By the homogeneity property, it must hold that if $M$ is nonzero and graded-homogeneous of degree $s$, $\|I(M)\|_1 = s$.

For $s \geq 1$, let $\ones_s$ denote the unique multiset with $\|\ones_s\|_1 = s$ and $\|\ones_s\|_\infty = 1$ (and $\ones_0 = \emptyset$). For each $n,m \leq k_1$, let $\tau_{n,m}(x,y) := x_{1,n}y_{1,m} - x_{1,m}y_{1,n}$. A graded-homogeneous polynomial $P$ of degree $s \geq 2$ is of \emph{$\tau$-type} if $P(x,y) = \tau_{n,m}(x,y)M'(x,y)$ for some $n,m \leq k_1$ and graded-homogeneous monomial $M'$ with $I(M') = \ones_{s-2}$.

A graded homogeneous polynomial of degree $s \geq 2$ of the form $\sum_j Q_j$ (the sum is finite), where each $Q_j$ is of $\tau$-type or a graded-homogeneous monomial of degree $s$ with $1 < \|I(Q_j)\|_\infty < s$ is called a \emph{BCH polynomial of degree $s$}.
\end{definition}

\begin{remark} \label{rmk:sametype}
Obviously a sum of BCH polynomials of degree $s$ is another such polynomial. If $P$ is a BCH polynomial of degree $s$, $1 \leq i \leq r$, and $1 \leq j \leq k_i$, $x_{i,j}P(x,y)$ and $y_{i,j}P(x,y)$ are BCH polynomials of degree $s+i$. If $P(x,y)$ is a BCH polynomial of degree $s$, then so is $P(x,\delta_t(y))$ for any $t \in \R_{>0}$.
\end{remark}

\begin{example}
Let $M(x,y) = 6x_{1,6}x_{1,1}^2y_{4,3}, P(x,y) = -y_{1,2}(x_{1,1}y_{1,2}-x_{1,2}y_{1,1})$, $Q(x,y) = x_{1,1}y_{1,1}$, and $R(x,y) = y_{3,2}$. $M$ is a graded-homogeneous monomial of degree $7$ with $I(M) = \{1,1,1,4\}$, $P$ is a graded homogeneous polynomial of degree $3$ of $\tau$-type, $Q$ is a graded-homogeneous monomial of degree $2$ with $I(Q) = \{1,1\}$, and $R$ is a graded homogeneous monomial of degree $3$ with $I(r) = \{3\}$. $M$ and $P$ are BCH polynomials, but $Q$ and $R$ are not because they are monomials with $\|I(Q)\|_\infty = 1$ and $\|I(R)\|_\infty = \|I(R)\|_1$.
\end{example}

We now arrive at a key structural lemma for the group product on graded nilpotent Lie algebras. The rest of this subsection is dedicated to its proof.

\begin{lemma} \label{lem:groupstruct}
For all $x,y \in \g$ and $2 \leq s \leq r$, \\ \\
$\begin{array}{llll}
(1) & (y^{-1}x)_1 &=& x_1-y_1 \\
(2) & (y^{-1}x)_s &=& x_s-y_s +  \sum_{j=1}^{k_s}P_{s,j}(x,y)U_{s,j}
\end{array}$ \\ \\
where each $P_{s,j}$ is a BCH polynomial of degree $s$.
\end{lemma}

A trusting reader familiar with the group structure of graded nilpotent Lie algebras induced by the Baker-Campbell-Hausdorff formula may safely skip the rest of this subsection. Before proving the lemma, we need to set some useful notation that allows us to work with nested Lie brackets, and then prove a lemma about these brackets.

\begin{definition}
Given $x,y \in \g$, $i \geq 1$, and $\eps \in \{1,2\}^i$, we recursively define $(x,y)^\eps$ as follows: for $i = 1$, $(x,y)^\eps := x$ if $\eps = 1$ and $(x,y)^\eps := y$ if $\eps = 2$. Assume $(x,y)^\eps$ has been defined for all $\eps \in \{1,2\}^i$ for some $i \geq 1$. Let $\eps \in \{1,2\}^{i+1}$. Then $\eps$ equals $(1,\eps')$ or $(2,\eps')$ for some $\eps' \in \{1,2\}^i$. We define $(x,y)^\eps := [x,(x,y)^{\eps'}]$ if $\eps = (1,\eps')$ and $(x,y)^\eps := [y,(x,y)^{\eps'}]$ if $\eps = (2,\eps')$.
\end{definition}

\begin{example}
$(x,y)^{(1,2,2,1)} = [x,[y,[y,x]]]$. The 1 or 2 in the superscript should be thought of as indicating the first or second component of $(x,y)$ in the nested Lie bracket.
\end{example}

\begin{lemma} \label{lem:bracketstruct}
For all $x,y \in \g$, $2 \leq i_1,i_2 \leq r$, and $\eps \in \{1,2\}^{i_1}$,
$$((x,y)^\eps)_{i_2} = \sum_{j=1}^{k_{i_2}}Q_{i_2,j}(x,y) U_{i_2,j}$$
where each $Q_{i_2,j}$ is a BCH polynomial of degree $i_2$ if $i_1 \leq i_2$ 0 if $i_1 > i_2$.
\end{lemma}

\begin{proof}
Let $x,y \in \g$. By the grading property, $((x,y)^{\eps})_{i_2} = 0$ if $\eps \in \{1,2\}^{i_1}$ and $i_1 > i_2$. We'll prove the remaining case by induction on $i_1$.

\underline{Proof of base case}. The base case is $i_1=2$. Let $\eps \in \{1,2\}^2$. Then $\eps$ equals $(1,1)$, $(1,2)$, $(2,1)$, or $(2,2)$. Since $(x,y)^{(1,1)} = (x,y)^{(2,2)} = 0$ and $(x,y)^{(2,1)} = -(x,y)^{(1,2)}$, it suffices to only consider $\eps = (1,2)$, in which case $(x,y)^\eps = [x,y]$. Let $i_2 \geq 2$. We treat the two cases $i_2 = 2$ and $i_2 > 2$. First assume $i_2 = 2$. Then we have
$$[x,y]_2 = [x_1,y_1] = \left[\sum_{j=1}^{k_1} x_{1,j} U_{1,j}, \sum_{j'=1}^{k_1} y_{1,j'} U_{1,j'}\right] = \sum_{j=1}^{k_1} \sum_{j'=1}^{k_{1}} x_{1,j}y_{1,j'}\left[U_{1,j},U_{1,j'}\right]$$
$$= \frac{1}{2} \left(\sum_{n,m=1}^{k_1} x_{1,n}y_{1,m}\left[U_{1,n},U_{1,m}\right] + \sum_{n,m=1}^{k_1} x_{1,m}y_{1,n}\left[U_{1,m},U_{1,n}\right]\right)$$
$$= \frac{1}{2}\sum_{n,m=1}^{k_1} (x_{1,n}y_{1,m}-x_{1,m}y_{1,n})\left[U_{1,n},U_{1,m}\right] = \frac{1}{2}\sum_{n,m=1}^{k_1} \tau_{n,m}(x,y)\left[U_{1,n},U_{1,m}\right]$$
$$= \frac{1}{2}\sum_{n,m=1}^{k_1} \tau_{n,m}(x,y) \sum_{j=1}^{k_2} c_{j,n,m}U_{2,j} = \sum_{j=1}^{k_2}\left(\sum_{n,m=1}^{k_1} \frac{c_{j,n,m}}{2}\tau_{n,m}(x,y)\right)U_{2,j}$$
for some $c_{j,n,m} \in \R$. The inner sum is a sum of polynomials of degree 2 of $\tau$-type, and thus a BCH polynomial of degree $i_2$.

Now we consider the case $i_2 > 2$.
$$[x,y]_{i_2} = \sum_{n=1}^{i_2-1}[x_n,y_{i_2-n}] = \sum_{n=1}^{i_2-1} \left[\sum_{j=1}^{k_n} x_{n,j} U_{n,j}, \sum_{j'=1}^{k_{i_2-n}} y_{i_2-n,j'} U_{i_2-n,j'}\right]$$
$$= \sum_{n=1}^{i_2-1} \sum_{j=1}^{k_n} \sum_{j'=1}^{k_{i_2-n}} x_{n,j}y_{i_2-n,j'}  \left[U_{n,j},U_{i_2-n,j'}\right] = \sum_{n=1}^{i_2-1} \sum_{j=1}^{k_n} \sum_{j'=1}^{k_{i_2-n}} x_{n,j}y_{i_2-n,j'} \sum_{m=1}^{k_{i_2}} c_{m,n,j,j'}U_{i_2,m}$$
$$= \sum_{m=1}^{k_{i_2}}\left(\sum_{n=1}^{i_2-1} \sum_{j=1}^{k_n} \sum_{j'=1}^{k_{i_2-n}} c_{m,n,j,j'}x_{n,j}y_{i_2-n,j'}\right)U_{i_2,m}$$
for some $c_{m,n,j,j'} \in \R$. Notice that, for each $n,j,j'$, $I(x_{n,j}y_{i_2-n,j'}) = \{n,i_2-n\}$, and so since $i_2 > 2$ and $1 \leq n \leq i_2-1$, $1 < \|I(x_{n,j}y_{i_2-n,j'})\|_\infty < i_2$, and thus $x_{n,j},y_{i_2-n,j'}$ is a BCH polynomial of degree $i_2$. This completes the proof of the base case.

\underline{Proof of inductive step}. Now assume the lemma holds for some $2 \leq i_1 < r$. Let $\eps \in \{1,2\}^{i_1+1}$. Then $\eps$ equals $(1,\eps')$ or $(2,\eps')$ for some $\eps' \in \{1,2\}^{i_1}$. Without loss of generality, assume $\eps = (1,\eps')$. Let $i_2 \geq i_1 + 1$.
Then
$$((x,y)^{\eps})_{i_2} = [x,(x,y)^{\eps'}]_{i_2} = \sum_{n=1}^{i_2-1}[x_n,((x,y)^{\eps'})_{i_2-n}]$$
$$\overset{\text{ind hyp}}{=} \sum_{n=1}^{i_2-1} \left[\sum_{j=1}^{k_n} x_{n,j} U_{n,j}, \sum_{j'=1}^{k_{i_2-n}}P_{i_2-n,j'}(x,y)U_{i_2-n,j'}\right]$$
$$= \sum_{n=1}^{i_2-1} \sum_{j=1}^{k_n} \sum_{j'=1}^{k_{i_2-n}} x_{n,j}P_{i_2-n,j'}(x,y) \left[U_{n,j},U_{i_2-n,j'}\right]$$
$$= \sum_{n=1}^{i_2-1} \sum_{j=1}^{k_n} \sum_{j'=1}^{k_{i_2-n}} x_{n,j}P_{i_2-n,j'}(x,y) \sum_{m=1}^{k_{i_2}}c_{m,n,j,j'}U_{i_2,m}$$
$$= \sum_{m=1}^{k_{i_2}} \left(\sum_{n=1}^{i_2} \sum_{j=1}^{k_n} \sum_{j'=1}^{k_{i_2-n}} c_{m,n,j,j'}x_{n,j}P_{i_2-n,j'}(x,y)\right)U_{i_2,m}$$
for some $c_{m,n,j,j'} \in \R$ and BCH polynomials $P_{i_2-n,j',\ell}$ of degree $i_2-n$. This implies $x_{n,j}P_{i_2-n,j'}(x,y)$ is a BCH polynomial of degree $i_2$, as desired.
\end{proof}

\begin{proof}[Proof of Lemma \ref{lem:groupstruct}]
The Baker-Campbell-Hausdorff formula, \eqref{eq:BCH}, implies that there are constants (many can be taken to be 0) $\{\alpha_{\eps}\}_{\eps \in \cup_{i=2}^r\{1,2\}^i} \sbs \R$ such that
$$y^{-1}x = x - y + \sum_{i=2}^r \sum_{\eps \in \{1,2\}^i} \alpha_\eps(x,y)^\eps$$
Since 
$$(y^{-1}x)_i = x_i - y_i + \sum_{i=2}^r \sum_{\eps \in \{1,2\}^i} \alpha_\eps((x,y)^\eps)_i$$
the desired conclusion follows by appealing to Lemma \ref{lem:bracketstruct}. 
\end{proof}

\subsection{Convex Metrics} \label{ss:upperbound}
The goal of this subsection is to prove Theorem \ref{thm:upperboundmain}. To do so, we construct a left invariant homogeneous quasi-metric on $\g$ that satisfies a certain 4-point inequality. This is the content of Lemma \ref{lem:convmetric}. All the lemmas and definitions preceding Lemma \ref{lem:convmetric} exist to prove it.

We next define a graded-homogeneous polynomial of degree $2s$ that dominates the square of any BCH polynomial of degree $s$, Lemma \ref{lem:domBCH1}. As a consequence of this we get two domination inequalities involving norms of group products, Lemmas \ref{lem:dompoly2} and \ref{lem:dompoly3}. These types of domination are what will ultimately allow us to prove Lemma \ref{lem:convnorm1}, the key lemma used in the proof of Lemma \ref{lem:convmetric}.

\begin{definition} \label{def:Ddef}
Let
$$\boldsymbol{\tau}(x,y) := \sqrt{\sum_{n,m \leq k_1} \tau^2_{n,m}(x,y)}$$
so that $\boldsymbol{\tau}(x,y)^2 \geq \tau_{n,m}(x,y)^2$ for every $n$ and $m$. For each $2 \leq s \leq r$, define $D_s: \g \cross \g \to \R_{\geq 0}$ recursively by
$$D_2(x,y) := \boldsymbol{\tau}^2(x,y) + \|(x_2,y_2)\|^2$$
$$D_{s+1} := \|(x_{s+1},y_{s+1})\|^2 + \sum_{s'=1}^{\lfloor (s+1)/2 \rfloor} \|(x_{s'},y_{s'})\|^2D_{s+1-s'}(x,y)$$
\end{definition}

\begin{lemma} \label{lem:domBCH1}
For any $2 \leq s \leq r$ and BCH polynomial $P$ of degree $s$, there exists $0 < c \leq 1$ such that for all $x,y \in \g$,
$$D_s(x,y) - \|(x_s,y_s)\|^2 \geq cP^2(x,y)$$
\end{lemma}

\begin{proof}
The proof is by induction on $s$. The base case $s=2$ is clear from the definition of $D_2$ and BCH polynomial of degree $2$. Assume the inequality holds for all $s_0 \leq s$ for some $s < r$. Let $P$ be a BCH polynomial of degree $s+1$. By definition of BCH polynomial, it suffices to prove the inequality assuming $P$ is a monomial with $1 < \|I(P)\|_\infty < s+1$ or $P$ is of $\tau$-type. First assume $P$ is a monomial with $1 < \|I(P)\|_\infty < s+1$. There are two subcases to consider: $1 \in I(P)$ and $1 \notin I(P)$. Assume the first subcase holds. Then $P = x_{1,n}M(x,y)$ or $P = y_{1,n}M(x,y)$ for some $n \leq k_1$ and monomial $M$ of degree $s$ with $1 < \|I(M)\|_\infty < s+1$. Then
$$D_{s+1}(x,y) - \|(x_{s+1},y_{s+1})\|^2 = \sum_{s'=1}^{\lfloor (s+1)/2 \rfloor} \|(x_{s'},y_{s'})\|^2D_{s+1-s'}(x,y)$$
$$\geq \|(x_1,y_1)\|^2D_{s}(x,y) \overset{\text{ind hyp}}{\geq} c\|(x_1,y_1)\|^2M^2(x,y) \geq cP^2(x,y)$$
Now assume the second subcase holds. Then $P(x,y) = x_{i,j}M(x,y)$ or $P(x,y) = y_{i,j}M(x,y)$ for some $1 < i \leq \lfloor (s+1)/2 \rfloor$, $j \leq k_i$, and monomial $M$ of degree $s+1-i$ with $1 < \|I(M)\|_\infty < s+1$. Then
$$D_{s+1}(x,y) - \|(x_{s+1},y_{s+1})\|^2 = \sum_{s'=1}^{\lfloor (s+1)/2 \rfloor} \|(x_{s'},y_{s'})\|^2D_{s+1-s'}(x,y)$$
$$\geq \|(x_i,y_i)\|^2D_{s+1-i}(x,y) \overset{\text{ind hyp}}{\geq} c\|(x_i,y_i)\|^2M^2(x,y) \geq cP^2(x,y)$$

Now assume $P$ is of $\tau$-type. By definition, since $P$ has degree $s+1$, this means $P(x,y) = \tau_{n,m}(x,y)M'(x,y)$ for some $n,m \leq k_1$ and graded-homogeneous monomial $M'$ with $I(M') = \ones_{s-1}$. This implies $P(x,y) = x_{1,\ell}P'(x,y)$ or $P(x,y) = y_{1,\ell}P'(x,y)$ for some $\ell \leq k_1$, and degree $s$ polynomial $P'$ of $\tau$-type. Then
$$D_{s+1}(x,y) - \|(x_{s+1},y_{s+1})\|^2 = \sum_{s'=1}^{\lfloor (s+1)/2 \rfloor} \|(x_{s'},y_{s'})\|^2D_{s+1-s'}(x,y)$$
$$\geq \|(x_1,y_1)\|^2D_{s}(x,y) \overset{\text{ind hyp}}{\geq} c\|(x_1,y_1)\|^2(P')^2(x,y) \geq cP^2(x,y)$$
\end{proof}

\begin{lemma} \label{lem:dompoly1}
Let $2 \leq s \leq r$. For any $t > 0$, there is a constant $c > 0$ such that for all $x,y \in \g$,
$$D_s(x,y) - \|(x_s,y_s)\|^2 \geq c\|(\delta_t(y)^{-1}x)_{s} - (x_{s}-t^{s}y_{s})\|^2$$
\end{lemma}

\begin{proof}
Let $t > 0$. By Lemma \ref{lem:groupstruct},
$$\|(\delta_t(y)^{-1}x)_s - (x_s-t^sy_s)\|^2 \overset{\text{Lem } \ref{lem:groupstruct}}{=} \sum_j |P_{s,j}(x,\delta_t(y))|^2 = \sum_j |P'_{s,j,t}(x,y)|^2$$
where each $P_{s,j}$ is a BCH polynomial of degree $s$, and by Remark \ref{rmk:sametype}, each $P'_{s,j,t}$ is a BCH polynomial of degree $s$. Then the desired inequality follows from Lemma \ref{lem:domBCH1}.
\end{proof}

\begin{lemma} \label{lem:dompoly2}
Let $2 \leq s \leq r$ and $c > 0$. For all sufficiently small $\lambda > 0$ (depending on $c$), for all $x,y \in \g$,
$$c(D_s(x,y)-\|(x_s,y_s)\|^2) + \lambda\|(y^{-1}x)_{s}\|^2 \geq \frac{\lambda}{2}\|x_s-y_s\|^2$$
\end{lemma}

\begin{proof}
Let $\lambda > 0$. By Lemma \ref{lem:dompoly1}, there is a constant $c' > 0$ (independent of $x,y$) such that
$$c(D_s(x,y)-\|(x_s,y_s)\|^2) + \lambda\|(y^{-1}x)_{s}\|^2 \geq c'\|(y^{-1}x)_{s} - (x_{s}-y_{s})\|^2 + \lambda\|(y^{-1}x)_{s}\|^2 =: (*)$$
Thus, if $\lambda \leq c'$,
$$(*) \geq \lambda\|(y^{-1}x)_{s} - (x_{s}-y_{s})\|^2 + \lambda\|(y^{-1}x)_{s}\|^2 \geq \frac{\lambda}{2}\|x_{s}-y_{s}\|^2$$
where the last inequality follows from the parallelogram law.
\end{proof}

\begin{lemma} \label{lem:dompoly3}
Let $2 \leq s \leq r$. There is a constant $c > 0$ such that for all $x,y \in \g$,
$$D_{s}(x,y) \geq c\|(\delta_{1/2}(y)^{-1}x)_{s}\|^2$$
\end{lemma}

\begin{proof}
By Lemma \ref{lem:dompoly1}, it suffices to show
$$\|(\delta_{1/2}(y)^{-1}x)_{s} - (x_{s}-2^{-s}y_{s})\|^2 + \|(x_{s},y_{s})\|^2 \geq \frac{1}{8}\|(\delta_{1/2}(y)^{-1}x)_{s}\|^2$$
Since
$$\|x_{s}-2^{-s}y_{s}\|^2 \leq 4\|(x_{s},y_{s})\|^2$$
it suffices to show
$$\|(\delta_{1/2}(y)^{-1}x)_{s} - (x_{s}-2^{-s}y_{s})\|^2 + \frac{1}{4}\|x_{s}-2^{-s}y_{s}\|^2 \geq \frac{1}{8}\|(\delta_{1/2}(y)^{-1}x)_{s}\|^2$$
This inequality is true by the parallelogram law.
\end{proof}

\begin{lemma} \label{lem:domtau}
There is a constant $c > 0$ such that for all $x,y \in \g$,
$$\|y_1\|\|x_1-y_1/2\| \geq c\boldsymbol{\tau}(x,y)$$
\end{lemma}

\begin{proof}
It suffices to show, for each fixed $n,m \leq k_1$, $\|y_1\|\|x_1-y_1/2\| \geq |\tau_{n,m}(x,y)|$. By Cauchy-Schwarz,
$$\|y_1\|\|x_1-y_1/2\| \geq \|(y_{1,m},-y_{1,n})\|\|(x_{1,n},x_{1,m})-(y_{1,n},y_{1,m})/2\|$$
$$\overset{\text{C-S}}{\geq} |y_{1,m}(x_{1,n}-y_{1,n}/2) - y_{1,n}(x_{1,m}-y_{1,m}/2)| = |x_{1,n}y_{1,m} - x_{1,m}y_{1,n}| = |\tau_{n,m}(x,y)|$$
\end{proof}

\begin{definition} \label{def:SNdef}
For $2 \leq s \leq r$, define $SN_s: \g \cross \g \to \R$ by
$$SN_s(x,y) := \max\{\|x_1-y_1/2\|,\|(x_2,y_2)\|^{1/2},\|(x_3,y_3)\|^{1/3}, \dots \|(x_s,y_s)\|^{1/s}\}$$
\end{definition}

\begin{remark}
Using the maximum of the terms is not important here; it could be replaced by any $\ell^p$-sum or other such norm. If a different choice of norm was used, the rest of the section would proceed the exact same way except with possibly different values of constants (but still independent of $x,y$).
\end{remark}

\begin{lemma} \label{lem:convnorm1}
For each $2 \leq s \leq r$, there exists a homogeneous quasi-norm $N_s$ and a constant $c > 0$ such that for all $x,y \in \g$,
\begin{enumerate}
\item \label{convnorm1} $$(N_s(x)^{2s}+N_s(y^{-1}x)^{2s})/2 - (N_s(y)/2)^{2s} \geq cSN_s^{2s}(x,y) + cD_s(x,y) + cN_s(\delta_{1/2}(y)^{-1}x)^{2s}$$
\item \label{convnorm2} $N_s(y) \geq \|y_1\|$ for all $s \geq 2$. Consequently, $N_{s}(y) + SN_s(x,y) \geq b\|(x_{s'},y_{s'})\|^{1/s'}$ for some $b > 0$ and all $1 \leq s' \leq s$.
\item \label{convnorm3} If $N_s(y) = 0$, $y_i = 0$ for all $1 \leq i \leq s$. In particular, $N_r$ is a positive definite homogeneous quasi-norm.
\end{enumerate}
\end{lemma}

\begin{proof}
The proof is by induction on $s$. The functions $N_s$ we construct will clearly be homogeneous quasi-norms and satisfy (\ref{convnorm2}) and (\ref{convnorm3}), so we will only concern ourselves with proving (\ref{convnorm1}).

\underline{Proof of base case}: The base case is $s = 2$. Throughout the proof of the base case, $c',c'',c'''$ denote (small) positive constants that depend on $\g$ but not on $x,y$. Each of the constants may depend on the ones previously appearing, but of course this is compatible with the fact that they are all independent of $x,y$. Define
$$N_2(x) := \sqrt[4]{\|x_1\|^4 + \lambda\|x_2\|^2}$$
where $\lambda > 0$ is to be chosen later. Recalling that $SN_2(x,y)^4 = \max(\|x_1-y_2/2\|^4,\|(x_2,y_2)\|^2) \leq \|x_1-y_2/2\|^4 + \|(x_2,y_2)\|^2$ and $D_2(x,y) = \boldsymbol{\tau}^2(x,y) + \|(x_2,y_2)\|^2$, we need to show
$$(N_2(x)^{4}+N_2(y^{-1}x)^{4})/2$$
$$\geq (N_2(y)/2)^{4} + c\|x_1-y_1/2\|^{4} + c\boldsymbol{\tau}^2(x,y) + c\|(x_2,y_2)\|^2 + cN_2(\delta_{1/2}(y)^{-1}x)^4$$
for some $\lambda, c > 0$. First let's write out the definitions of some of the terms in the inequality. \\ \\
$\begin{array}{lll}
N_2(x)^4 &=& \|x_1\|^4 + \lambda\|x_2\|^2 \\
N_2(y^{-1}x)^4 &=& \|x_1-y_1\|^4 + \lambda\|(y^{-1}x)_2\|^2 \\
N_2(y)^4 &=& \|y_1\|^4 + \lambda\|y_2\|^2 \\
\end{array}$ \\ \\
By convexity, parallelogram law, and Lemma \ref{lem:domtau},
$$(\|x_1\|^4 + \|x_1-y_1\|^4)/2 \geq ((\|x_1\|^2 + \|x_1-y_1\|^2)/2)^2$$
$$= (\|y_1/2\|^2 + \|x_1-y_1/2\|^2)^2 = (\|y_1\|/2)^4 + 2\|y_1/2\|^2\|x_1-y_1/2\|^2 + \|x_1-y_1/2\|^4$$
$$\overset{\text{Lem }\ref{lem:domtau}}{\geq} (\|y_1\|/2)^4 + c'\boldsymbol{\tau}^2(x,y) + \|x_1-y_1/2\|^4$$
For some $c' > 0$. Thus, it suffices to show that for sufficiently small $\lambda,c > 0$, 
$$\frac{c'}{2}\boldsymbol{\tau}^2(x,y) + \lambda\|x_2\|^2 + \lambda\|(y^{-1}x)_2\|^2 + \frac{1}{2}\|x_1-y_1/2\|^4$$
$$\geq 2^{-4}\lambda\|y_2\|^2 + c\boldsymbol{\tau}^2(x,y) + c\|(x_2,y_2)\|^2 + cN_2(\delta_{1/2}(y)^{-1}x)^4$$
By Lemma \ref{lem:dompoly2}, the following inequality is true for sufficiently small $\lambda > 0$:
$$\frac{c'}{4}\boldsymbol{\tau}^2(x,y) + \lambda\|(y^{-1}x)_2\|^2 = \frac{c'}{4}(D_2(x,y) - \|(x_2,y_2)\||^2) + \lambda\|(y^{-1}x)_2\|^2 \overset{\text{Lem }\ref{lem:dompoly2}}{\geq} \frac{\lambda}{2}\|x_2-y_2\|^2$$
Thus it suffices for the following inequality to hold for $\lambda,c > 0$ sufficiently small:
$$\frac{c'}{4}\boldsymbol{\tau}^2(x,y) + \lambda\|x_2\|^2 + \frac{\lambda}{2}\|x_2-y_2\|^2 + \frac{1}{2}\|x_1-y_1/2\|^4$$
$$\geq 2^{-4}\lambda\|y_2\|^2 + c\|(x_2,y_2)\|^2 + cN_2(\delta_{1/2}(y)^{-1}x)^4$$
We have
$$\lambda\|x_2\|^2 + \frac{\lambda}{2}\|x_2-y_2\|^2 \geq \frac{\lambda}{2}(\|x_2\|^2 + \|x_2-y_2\|^2)$$
$$= \frac{\lambda}{4}(\|x_2\|^2 + \|x_2-y_2\|^2) + \frac{\lambda}{4}(\|x_2\|^2 + \|x_2-y_2\|^2) \geq 2^{-4}\lambda\|y_2\|^2 + c''\|(x_2,y_2)\|^2$$
Thus it remains to show
$$\frac{c'}{4}\boldsymbol{\tau}^2(x,y) + \frac{c''}{2}\|(x_2,y_2)\|^2 + \frac{1}{2}\|x_1-y_1/2\|^4 \geq cN_2(\delta_{1/2}(y)^{-1}x)^4$$
for $c > 0$ sufficiently small.
By Lemma \ref{lem:dompoly3} we have
$$\frac{c'}{4}\boldsymbol{\tau}^2(x,y) + \frac{c''}{2}\|(x_2,y_2)\|^2 \geq c'''(\boldsymbol{\tau}^2(x,y) + \|(x_2,y_2)\|^2) = c'''D_s(x,y) \overset{\text{Lem }\ref{lem:dompoly3}}{\geq} c\lambda\|(\delta_{1/2}(y)^{-1}x)_2\|^2$$
$c > 0$ sufficiently small, and thus it remains to show
$$c\lambda\|(\delta_{1/2}(y)^{-1}x)_2\|^2 + \frac{1}{2}\|x_1-y_1/2\|^4 \geq cN_2(\delta_{1/2}(y)^{-1}x)^4$$
This is true by definition of $N_2$. This completes the proof of the base case.

\underline{Proof of inductive step}: Now assume the statement holds for all $2 \leq s' \leq s$ some $2 \leq s \leq r-1$. Define $N_{s+1}$ by
\begin{equation} \label{eq:Ndef}
N_{s+1}(x) := \sqrt[2(s+1)]{\lambda\|x_{s+1}\|^2 + \sum_{s'=\lceil (s+1)/2 \rceil}^s N_{s'}^{2(s+1)}(x)}
\end{equation}
where $\lambda$ is a (small) positive constant (different $\lambda$ than in the base case) to be chosen later (independent of $x,y$). Throughout the remainder of the proof, $c_1-c_7$ denote (small) positive constants that depend on $\g$ but not on $x,y$. Each of the constants may depend on the ones previously appearing, but of course this is compatible with the fact that they are all independent of $x,y$. The constant $\lambda$ will end up depending on $c_2$ (which in turn depends on $c_1$), and the subsequent constants will depend on $\lambda$.

We now prove the inductive step. In what follows, we adopt some conventions to help make the proof more readable. There are two types of equalities/inequalities we use relating each of the expressions below. The first type is simply using a lemma, definition, inductive hypothesis, or convexity or trivial numerical inequality. Whenever an equality/inequality of this type is used, the particular terms in the expression that change from one to the next are \textbf{bolded}. No other terms change, except for the bolded ones to which the particular lemma, definition, inductive hypothesis, or convexity or trivial numerical inequality apply. Apart from the trivial numerical inequalities, the name of the lemma or definition, ``ind hyp", or ``convexity" decorates the equality/inequality symbol. The second type of equality/inequality used is always an equality and the equality symbol is decorated with the word ``rearrange". This means we use trivialities like commutivity of addition or multiplication, reindexing of a sum, or no symbolic changes at all. Importantly, we also use equalities decorated with ``rearrange" to change which terms are bolded in the expression, in preparation for the use of another equality/inequality of the first type. \\ \\
$\begin{array}{lc}
& \boldsymbol{(N_{s+1}(x)^{2(s+1)}+N_{s+1}(y^{-1}x)^{2(s+1)})/2} \\ \\
     
\overset{\eqref{eq:Ndef}}{=} & \boldsymbol{\left(\lambda\|x_{s+1}\|^2 + \sum_{s'= \lceil (s+1)/2 \rceil}^s N_{s'}^{2(s+1)}(x) + \|(y^{-1}x)_{s+1}\|^2\right.} \\
& \boldsymbol{\left.+ \sum_{s'= \lceil (s+1)/2 \rceil}^s N_{s'}^{2(s+1)}(y^{-1}x)\right)/2} \\ \\

\overset{\text{rearrange}}{=} & \dfrac{\lambda}{2}(\|x_{s+1}\|^2 + \|(y^{-1}x)_{s+1}\|^2) + \boldsymbol{\left(\sum_{s'= \lceil (s+1)/2 \rceil}^s N_{s'}^{2(s+1)}(x) + N_{s'}^{2(s+1)}(y^{-1}x)\right)/2} \\ \\

\overset{\text{convexity}}{\geq} & \dfrac{\lambda}{2}(\|x_{s+1}\|^2 + \|(y^{-1}x)_{s+1}\|^2) + \boldsymbol{\sum_{s'= \lceil (s+1)/2 \rceil}^s \left(\dfrac{N_{s'}^{2s'}(x) + N_{s'}^{2s'}(y^{-1}x)}{2}\right)^{\frac{s+1}{s'}}}
\end{array}$ \\ \\

$\begin{array}{lc}
\overset{\text{ind hyp }(\ref{convnorm1})}{\geq} & \dfrac{\lambda}{2}(\|x_{s+1}\|^2 + \|(y^{-1}x)_{s+1}\|^2) \\
& + \sum_{s'= \lceil (s+1)/2 \rceil}^s \boldsymbol{((N_{s'}(y)/2)^{2s'} + c_1SN_{s'}^{2s'}(x,y)} \\
& \boldsymbol{+ c_1D_{s'}(x,y) + c_1N_{s'}(\delta_{1/2}(y)^{-1}x)^{2s'})^{\frac{s+1}{s'}}} \\ \\

\overset{\text{Lem }\ref{lem:binomineq}}{\geq} & \dfrac{\lambda}{2}(\|x_{s+1}\|^2 + \|(y^{-1}x)_{s+1}\|^2) \\
& + \sum_{s'= \lceil (s+1)/2 \rceil}^s \boldsymbol{((N_{s'}(y)/2)^{2s'} + c_1SN_{s'}^{2s'}(x,y) +  c_1N_{s'}(\delta_{1/2}(y)^{-1}x)^{2s'})^{\frac{s+1}{s'}}} \\
& \boldsymbol{+ ((N_{s'}(y)/2)^{2s'} + c_1SN_{s'}^{2s'}(x,y))^{\frac{s+1-s'}{s'}}c_1D_{s'}(x,y)} \\ \\

\overset{\text{rearrange}}{=} & \dfrac{\lambda}{2}(\|x_{s+1}\|^2 + \|(y^{-1}x)_{s+1}\|^2) \\
& + \sum_{s'= \lceil (s+1)/2 \rceil}^s \boldsymbol{((N_{s'}(y)/2)^{2s'} + c_1SN_{s'}^{2s'}(x,y) +  c_1N_{s'}(\delta_{1/2}(y)^{-1}x)^{2s'})^{\frac{s+1}{s'}}} \\
& + ((N_{s'}(y)/2)^{2s'} + c_1SN_{s'}^{2s'}(x,y))^{\frac{s+1-s'}{s'}}c_1D_{s'}(x,y) \\ \\

\overset{\text{convexity}}{\geq} & \dfrac{\lambda}{2}(\|x_{s+1}\|^2 + \|(y^{-1}x)_{s+1}\|^2) \\
& + \sum_{s'= \lceil (s+1)/2 \rceil}^s \boldsymbol{(N_{s'}(y)/2)^{2(s+1)} + c_1SN_{s'}^{2(s+1)}(x,y) +  c_1N_{s'}(\delta_{1/2}(y)^{-1}x)^{2(s+1)}} \\
& + ((N_{s'}(y)/2)^{2s'} + c_1SN_{s'}^{2s'}(x,y))^{\frac{s+1-s'}{s'}}c_1D_{s'}(x,y) \\ \\

\overset{\text{rearrange}}{=} & \dfrac{\lambda}{2}(\|x_{s+1}\|^2 + \|(y^{-1}x)_{s+1}\|^2) \\
& + \sum_{s'= \lceil (s+1)/2 \rceil}^s (N_{s'}(y)/2)^{2(s+1)} + c_1SN_{s'}^{2(s+1)}(x,y) +  c_1N_{s'}(\delta_{1/2}(y)^{-1}x)^{2(s+1)} \\
& + \boldsymbol{((N_{s'}(y)/2)^{2s'} + c_1SN_{s'}^{2s'}(x,y))^{\frac{s+1-s'}{s'}}c_1}D_{s'}(x,y)
\end{array}$  \\ \\

$\begin{array}{lc}
\overset{\text{ind hyp }(\ref{convnorm2})}{\geq} & \dfrac{\lambda}{2}(\|x_{s+1}\|^2 + \|(y^{-1}x)_{s+1}\|^2) \\
& + \sum_{s'= \lceil (s+1)/2 \rceil}^s (N_{s'}(y)/2)^{2(s+1)} + c_1SN_{s'}^{2(s+1)}(x,y) +  c_1N_{s'}(\delta_{1/2}(y)^{-1}x)^{2(s+1)} \\
& + \boldsymbol{c_2\|(x_{s+1-s'},y_{s+1-s'})\|^2}D_{s'}(x,y)
\end{array}$ \\ \\

$\begin{array}{lc}
\overset{\text{rearrange}}{=} & \dfrac{\lambda}{2}(\|x_{s+1}\|^2 + \|(y^{-1}x)_{s+1}\|^2) \\
& + \boldsymbol{\sum_{s'= \lceil (s+1)/2 \rceil}^s} (N_{s'}(y)/2)^{2(s+1)} + c_1\boldsymbol{SN_{s'}^{2(s+1)}(x,y)} +  c_1N_{s'}(\delta_{1/2}(y)^{-1}x)^{2(s+1)} \\
& + \sum_{s'=1}^{\lfloor (s+1)/2 \rfloor} c_2\|(x_{s'},y_{s'})\|^2D_{s+1-s'}(x,y)
\end{array}$ \\ \\

$\begin{array}{lc}
\geq & \dfrac{\lambda}{2}(\|x_{s+1}\|^2 + \|(y^{-1}x)_{s+1}\|^2) \\
& + c_1\boldsymbol{SN_{s}^{2(s+1)}(x,y)} + \sum_{s'= \lceil (s+1)/2 \rceil}^s (N_{s'}(y)/2)^{2(s+1)} +  c_1N_{s'}(\delta_{1/2}(y)^{-1}x)^{2(s+1)} \\
& + \sum_{s'=1}^{\lfloor (s+1)/2 \rfloor} c_2\|(x_{s'},y_{s'})\|^2D_{s+1-s'}(x,y) \\ \\

\overset{\text{rearrange}}{=} & \dfrac{\lambda}{2}(\|x_{s+1}\|^2 + \|(y^{-1}x)_{s+1}\|^2) \\
& + c_1SN_{s}^{2(s+1)}(x,y) + \sum_{s'= \lceil (s+1)/2 \rceil}^s (N_{s'}(y)/2)^{2(s+1)} +  c_1N_{s'}(\delta_{1/2}(y)^{-1}x)^{2(s+1)} \\
& + \boldsymbol{\sum_{s'=1}^{\lfloor (s+1)/2 \rfloor} c_2\|(x_{s'},y_{s'})\|^2D_{s+1-s'}(x,y)}
\end{array}$ \\ \\

$\begin{array}{lc}
\overset{\text{Def }\ref{def:Ddef}}{=} & \dfrac{\lambda}{2}(\|x_{s+1}\|^2 + \|(y^{-1}x)_{s+1}\|^2) \\
& + c_1SN_{s}^{2(s+1)}(x,y) + \sum_{s'= \lceil (s+1)/2 \rceil}^s (N_{s'}(y)/2)^{2(s+1)} +  c_1N_{s'}(\delta_{1/2}(y)^{-1}x)^{2(s+1)} \\
& + \boldsymbol{c_2(D_{s+1}(x,y) - \|(x_{s+1},y_{s+1})\|^2)} \\ \\

\overset{\text{rearrange}}{=} & c_1SN_{s}^{2(s+1)}(x,y) + \sum_{s'= \lceil (s+1)/2 \rceil}^s (N_{s'}(y)/2)^{2(s+1)} +  c_1N_{s'}(\delta_{1/2}(y)^{-1}x)^{2(s+1)} \\
& + \dfrac{c_2}{2}(D_{s+1}(x,y) - \|(x_{s+1},y_{s+1})\|^2) \\
& \boldsymbol{\dfrac{c_2}{2}(D_{s+1}(x,y) - \|(x_{s+1},y_{s+1})\|^2) + \dfrac{\lambda}{2}(\|x_{s+1}\|^2 + \|(y^{-1}x)_{s+1}\|^2)} =: (*)
\end{array}$ \\ \\
By Lemma \ref{lem:dompoly2}, we can choose $\lambda > 0$ sufficiently small so that
$$
\frac{c_2}{2}(D_{s+1}(x,y) - \|(x_{s+1},y_{s+1})\|^2) + \frac{\lambda}{2}(\|x_{s+1}\|^2 + \|(y^{-1}x)_{s+1}\|^2) \overset{\text{Lem }\ref{lem:dompoly2}}{\geq} \frac{\lambda}{4}(\|x_{s+1}\|^2+\|x_{s+1}-y_{s+1}\|^2)$$
$$= \frac{\lambda}{8}(\|x_{s+1}\|^2+\|x_{s+1}-y_{s+1}\|^2) + \frac{\lambda}{8}(\|x_{s+1}\|^2+\|x_{s+1}-y_{s+1}\|^2)$$
$$\geq \frac{\lambda}{16}\|y_{s+1}\|^2 + c_3\|(x_{s+1},y_{s+1})\|^2 \geq 2^{-(s+1)}\lambda\|y_{s+1}\|^2 + c_3\|(x_{s+1},y_{s+1})\|^2$$
And thus we get \\ \\
$\begin{array}{lc}
(*) \geq & c_1SN_{s}^{2(s+1)}(x,y) + \sum_{s'= \lceil (s+1)/2 \rceil}^s (N_{s'}(y)/2)^{2(s+1)} +  c_1N_{s'}(\delta_{1/2}(y)^{-1}x)^{2(s+1)} \\
& + \dfrac{c_2}{2}(D_{s+1}(x,y) - \|(x_{s+1},y_{s+1})\|^2) \\
& \boldsymbol{2^{-(s+1)}\lambda\|y_{s+1}\|^2 + c_3\|(x_{s+1},y_{s+1})\|^2}
\end{array}$ \\ \\

$\begin{array}{lc}
\overset{\text{rearrange}}{=} & c_1SN_{s}^{2(s+1)}(x,y) + \boldsymbol{2^{-(s+1)}\lambda\|y_{s+1}\|^2 + \sum_{s'= \lceil (s+1)/2 \rceil}^s (N_{s'}(y)/2)^{2(s+1)}} \\
& + \sum_{s'= \lceil (s+1)/2 \rceil}^s c_1N_{s'}(\delta_{1/2}(y)^{-1}x)^{2(s+1)} + \dfrac{c_2}{2}(D_{s+1}(x,y) - \|(x_{s+1},y_{s+1})\|^2) \\
& + c_3\|(x_{s+1},y_{s+1})\|^2 \\ \\

\overset{\eqref{eq:Ndef}}{=} & c_1SN_{s}^{2(s+1)}(x,y) + \boldsymbol{(N_{s+1}(y)/2)^{2(s+1)}} \\
& + \sum_{s'= \lceil (s+1)/2 \rceil}^s c_1N_{s'}(\delta_{1/2}(y)^{-1}x)^{2(s+1)} + \dfrac{c_2}{2}(D_{s+1}(x,y) - \|(x_{s+1},y_{s+1})\|^2) \\
& + c_3\|(x_{s+1},y_{s+1})\|^2
\end{array}$ \\ \\

$\begin{array}{lc}
\overset{\text{rearrange}}{=} & \boldsymbol{c_1SN_{s}^{2(s+1)}(x,y) + \dfrac{c_3}{2}\|(x_{s+1},y_{s+1})\|^2} + (N_{s+1}(y)/2)^{2(s+1)} \\
& + \sum_{s'= \lceil (s+1)/2 \rceil}^s c_1N_{s'}(\delta_{1/2}(y)^{-1}x)^{2(s+1)} \\
& + \dfrac{c_2}{2}(D_{s+1}(x,y) - \|(x_{s+1},y_{s+1})\|^2) + \dfrac{c_3}{2}\|(x_{s+1},y_{s+1})\|^2 \\ \\

\overset{\text{Def  }\ref{def:SNdef}}{\geq} & \boldsymbol{c_4SN_{s+1}^{2(s+1)}(x,y)} + (N_{s+1}(y)/2)^{2(s+1)} \\
& + \sum_{s'= \lceil (s+1)/2 \rceil}^s c_1N_{s'}(\delta_{1/2}(y)^{-1}x)^{2(s+1)} \\
& + \dfrac{c_2}{2}(D_{s+1}(x,y) - \|(x_{s+1},y_{s+1})\|^2) + \dfrac{c_3}{2}\|(x_{s+1},y_{s+1})\|^2 \\ \\

\overset{\text{rearrange}}{=} & (N_{s+1}(y)/2)^{2(s+1)} + c_4SN_{s+1}^{2(s+1)}(x,y) + \sum_{s'= \lceil (s+1)/2 \rceil}^s c_1N_{s'}(\delta_{1/2}(y)^{-1}x)^{2(s+1)} \\
& + \boldsymbol{\dfrac{c_2}{2}(D_{s+1}(x,y) - \|(x_{s+1},y_{s+1})\|^2) + \dfrac{c_3}{2}\|(x_{s+1},y_{s+1})\|^2}
\end{array}$ \\ \\

$\begin{array}{lc}
\geq & (N_{s+1}(y)/2)^{2(s+1)} + c_4SN_{s+1}^{2(s+1)}(x,y) + \sum_{s'= \lceil (s+1)/2 \rceil}^s c_1N_{s'}(\delta_{1/2}(y)^{-1}x)^{2(s+1)} \\
& + \boldsymbol{c_5D_{s+1}(x,y)} \\ \\

\overset{\text{rearrange}}{=} & (N_{s+1}(y)/2)^{2(s+1)} + c_4SN_{s+1}^{2(s+1)}(x,y) + \dfrac{c_5}{2}D_{s+1}(x,y) \\
& + \sum_{s'= \lceil (s+1)/2 \rceil}^s c_1N_{s'}(\delta_{1/2}(y)^{-1}x)^{2(s+1)} + \boldsymbol{\dfrac{c_5}{2}D_{s+1}(x,y)}
\end{array}$ \\ \\

$\begin{array}{lc}
\overset{\text{Lem }\ref{lem:dompoly3}}{\geq} & (N_{s+1}(y)/2)^{2(s+1)} + c_4SN_{s+1}^{2(s+1)}(x,y) + \dfrac{c_5}{2}D_{s+1}(x,y) \\
& + \sum_{s'= \lceil (s+1)/2 \rceil}^s c_1N_{s'}(\delta_{1/2}(y)^{-1}x)^{2(s+1)} + \boldsymbol{c_6\|(\delta_{1/2}(y)^{-1}x)_{s+1}\|^2} \\ \\

\overset{\text{rearrange}}{=} & (N_{s+1}(y)/2)^{2(s+1)} + c_4SN_{s+1}^{2(s+1)}(x,y) + \dfrac{c_5}{2}D_{s+1}(x,y) \\
& + \boldsymbol{\sum_{s'= \lceil (s+1)/2 \rceil}^s c_1N_{s'}(\delta_{1/2}(y)^{-1}x)^{2(s+1)} + c_6\|(\delta_{1/2}(y)^{-1}x)_{s+1}\|^2} \\ \\

\overset{\eqref{eq:Ndef}}{\geq} & (N_{s+1}(y)/2)^{2(s+1)} + c_4SN_{s+1}^{2(s+1)}(x,y) + \dfrac{c_5}{2}D_{s+1}(x,y) \\
& + \boldsymbol{c_7N_{s+1}(\delta_{1/2}(y)^{-1}x)^{2(s+1)}}
\end{array}$ \\
\end{proof}

\begin{lemma} \label{lem:convnorm2}
There exists a positive definite homogeneous quasi-norm $N_r$ on $\g$ and a constant $c > 0$ (depending on $\g$ but not on $x,y$) such that for all $p \geq r$ and all $x,y \in \g$,
$$(N_r(x)^{2p}+N_r(y^{-1}x)^{2p})/2 - (N_r(y)/2)^{2p} \geq c^{p/r}N_r(\delta_{1/2}(y)^{-1}x)^{2p}$$
\end{lemma}

\begin{proof}
Let $N_r,c$ be as in the conclusion of Lemma \ref{lem:convnorm1}. Let $p \geq r$. Then by convexity and that lemma,
$$(N_r(x)^{2p}+N_r(y^{-1}x)^{2p})/2 \geq ((N_r(x)^{2r}+N_r(y^{-1}x)^{2r})/2)^{p/r}$$
$$\overset{\text{Lem }\ref{lem:convnorm1}}{\geq} ((N_r(y)/2)^{2r} + cN_r(\delta_{1/2}(y)^{-1}x)^{2r})^{p/r} \geq (N_r(y)/2)^{2p} + c^{p/r}N_r(\delta_{1/2}(y)^{-1}x)^{2p}$$
\end{proof}

\begin{lemma} \label{lem:convmetric}
There exists a left invariant, homogeneous, positive definite quasi-metric $d_{N_r}$ on $\g$ and a constant $c > 0$ (depending on $\g$ but not on $w,x,y,z$) such that for all $p \geq r$ and $w,x,y,z \in \g$, 
$$(2d_{N_r}(y,x)^{2p}+d_{N_r}(y,w)^{2p}+d_{N_r}(y,z)^{2p})/2 - (d_{N_r}(x,w)/2)^{2p} - (d_{N_r}(x,z)/2)^{2p} \geq c'd_{N_r}(w,z)^{2p}$$
\end{lemma}

\begin{proof}
Let $N_r,c$ be as in the previous lemma. Let $d_{N_r}$ be the metric derived from $N_r$; $d_{N_r}(x,y) := N_r(y^{-1}x)$. By left invariance of the metric, we may assume $x = 0$. Then by applying the previous lemma to each of the pairs $(y,w)$ and $(y,z)$, we obtain
$$(d_{N_r}(y,0)^{2p} + d_{N_r}(y,w)^{2p})/2 - (d_{N_r}(0,w)/2)^{2p} \geq c^{p/r}d_{N_r}(\delta_{1/2}(w),0)^{2p}$$
$$(d_{N_r}(y,0)^{2p} + d_{N_r}(y,z)^{2p})/2 - (d_{N_r}(0,z)/2)^{2p} \geq c^{p/r}d_{N_r}(\delta_{1/2}(z),0)^{2p}$$
Adding these and then using using H\"older, the quasi-triangle inequality, and homogeneity gives
$$(2d_{N_r}(y,0)^{2p} + d_{N_r}(y,w)^{2p} + d_{N_r}(y,z)^{2p})/2 - (d_{N_r}(0,w)/2)^{2p} - (d_{N_r}(0,z)/2)^{2p}$$
$$\geq c^{p/r}(d_{N_r}(\delta_{1/2}(w),0)^{2p} + d_{N_r}(\delta_{1/2}(z),0)^{2p}) \geq 2^{-2p+1}c^{p/r}(d_{N_r}(\delta_{1/2}(w),0) + d_{N_r}(\delta_{1/2}(z),0))^{2p}$$
$$\geq c'd_{N_r}(\delta_{1/2}(w),\delta_{1/2}(z))^{2p} = 2^{-2p}c'd_{N_r}(w,z)^{2p}$$
for some $c' > 0$.
\end{proof}

\begin{theorem} \label{thm:upperboundmain}
Every graded nilpotent Lie group of step $r$, equipped with a left invariant metric homogeneous with respect to the dilations induced by the grading, is Markov $p$-convex for every $p \in [2r,\infty)$.
\end{theorem}

\begin{proof}
Markov $p$-convexity is invariant under biLipschitz equivalence. Thus, we need only show $(\g,d_{N_r})$ is Markov $2p$-convex for all $p \geq r$, where $d_{N_r}$ is the quasi-metric from Lemma \ref{lem:convmetric}. The Markov convexity of $d_{N_r}$ follows from the 4-point inequality of Lemma \ref{lem:convmetric} and the proof of Proposition 2.1 in \cite{MN}.
\end{proof}

\section{Lower Bound on Markov Convexity of $J^{r-1}(\R)$}
The goal of this section is to prove Theorem \ref{thm:lowerboundmain}, which occurs at the conclusion. The strategy is to construct a sequence of directed graphs (see Definition \ref{def:graphs}) with bad Markov convexity properties. These bad properties are manifested by the dispersive nature of random walks on the graphs. This is the content of Lemma \ref{lem:badconvexity2}. We then map these graphs into $J^{r-1}(\R)$ with sufficient control over the distortion (Lemma \ref{lem:mapintojetspace}) to prove Theorem \ref{thm:lowerboundmain}.

\subsection{Directed Graphs and Random Walks}
Let $(N_m)_{m=0}^\infty$ be any sequence of integers with $N_0 = 0$ and $N_{m+1} \geq \max(1,N_m + \lceil 2\log_2(m+1) \rceil)$. We'll define a sequence of directed graphs $(\Gamma_m)_{m=0}^\infty$. The graphs will be directed from unique source vertex to unique and sink vertex, which we will denote by $0_m$ and $1_m$, respectively. Let diam$(\Gamma_m)$ be the number of edges in a directed edge path from $0$ to $1$, which is also equal to the diameter of $\Gamma_m$ with respect to the shortest path metric. The construction will be such that diam$(\Gamma_m) = 2^{N_m}$.

\begin{definition} \label{def:graphs}
We'll perform the construction and also prove that diam$(\Gamma_m) = 2^{N_m}$ by induction. Let $\Gamma_0$ be the interval $I$, that is, a graph with two vertices $0,1$ and a single edge connecting them, directed from $0$ to $1$. Suppose $\Gamma_m$ has been constructed for some $m \geq 0$. We define an intermediate graph $\Gamma_{m+1}'$ by gluing together $a := 2^{N_{m+1} - \lceil 2\log_2(m+1) \rceil - 1}$ copies of $I$, then $A:= 2^{N_{m+1}-N_m} - 2^{N_{m+1}-N_m- \lceil 2\log_2(m+1) \rceil} = 2^{-N_m}(2^{N_{m+1}}-2a) = 2^{N_{m+1}-N_m}(1-2^{-\lceil 2\log_2(m+1) \rceil})$ copies of $\Gamma_m$, then $a$ more copies of $I$ again together in series. The source vertex of this graph is the source vertex of the first copy of $I$, and the sink vertex is the sink vertex of the last copy of $I$. The diameter of this graph is
$$a \cdot \text{diam}(I) + A \cdot \text{diam}(\Gamma_m) + a \cdot \text{diam}(I) \overset{\text{ind hyp}}{=} 2a + 2^{N_m}A = 2^{N_{m+1}}$$
We then define $\Gamma_{m+1}$ to be two copies of $\Gamma_{m+1}'$, denoted $+\Gamma_{m+1}'$ and $-\Gamma_{m+1}'$, glued together in parallel. Denote the common source vertex $0_m$ and sink vertex $1_m$. The diameter of $\Gamma_{m+1}$ is the same as the diameter of $\Gamma_{m+1}'$. We note that each copy of $\Gamma_m$ in $\Gamma_{m+1}$ is isometrically embedded; any shortest path between two points in a copy of $\Gamma_m \sbs \Gamma_{m+1}$ completely belongs to $\Gamma_m$.

By swapping $+\Gamma_{m+1}'$ and $-\Gamma_{m+1}'$ in $\Gamma_{m+1}$, we obtain a directed graph involution $\iota: \Gamma_{m+1} \to \Gamma_{m+1}$.

For $q_1,q_2 \in \Gamma_m$, $(q_1,q_2)$ is called a \emph{vertical pair} if $d_m(q_1,0_m) = d_m(q_2,0_m)$.
\end{definition}

For each $m \geq 0$, let $(X^m_t)_{t=0}^{2^{N_m}}$ be the standard directed random walk on $\Gamma_m$. Let $d_m$ denote the shortest path metric on $\Gamma_m$. With full probability, $d(X^m_t,0_m) = t$ for $0 \leq t \leq 2^{N_m}$.

See the two right-hand graphs of Figure \ref{fig:Fm} for what $\Gamma_1$ and $\Gamma_2$ look like when $N_0 = 0$, $N_1 = 2$, and $N_2 = 4$. The graphs are drawn in such a way that the direction is from left to right, $+\Gamma_m'$ lies above the $x$-axis, and $-\Gamma_m'$ lies below the $x$-axis. The source vertices $0_m$ are both drawn at $(0,0)$, and the sink vertices $1_2$ are both drawn at $(1,0)$.

\begin{lemma} \label{lem:badconvexity1}
For all $p > 0$ and $m \geq 0$,
$$\sum_{k=0}^{N_m} \sum_{t=1}^{2^{N_m}} \frac{\E[d_m(X^m_t,\tilde{X}^m_t(t-2^k))^p]}{2^{kp}} \geq \frac{m}{8}2^{N_m}\Pi_{i=1}^{m-1}(1-(i+1)^{-2})$$
\end{lemma}

\begin{proof}
Let $p \geq 1$. The proof is by induction on $m$. The base case $m = 0$ is trivially true. Assume the inequality holds for some $m \geq 0$. Now we consider the standard random walk $X^{m+1}_t$ on $\Gamma_{m+1}$. Consider $k$ and $t$ in the range $a + 1 \leq t \leq 2^{N_{m+1}} - a$, $0 \leq k \leq N_m$, where $a = 2^{N_{m+1}-\lceil 2\log_2(m+1) \rceil -1}$. Then $t - 2^k \geq 2^{N_{m+1}- \lceil 2\log_2(m+1) \rceil -1} + 1 - 2^{N_m} \geq 1$, so $X^{m+1}_1$ and $\tilde{X}^{m+1}_1(t-2^k)$ agree. Then for all subsequent times, with full probability, $X^{m+1}_t$ and $\tilde{X}^{m+1}_t(t-2^k)$ belong to the same copy of $\Gamma_{m+1}'$ in $\Gamma_{m+1}$. Then, after recalling the construction of $\Gamma_{m+1}'$ as a number of copies of $\Gamma_m$ and $I$ glued together, it can be seen that for the range of $t$ in interest, $X^{m+1}_t$ and $\tilde{X}^{m+1}_t(t-2^k)$ are standard random walks across $A = 2^{N_{m+1}-N_m}(1-2^{-\lceil 2\log_2(m+1) \rceil})$ consecutive copies of $\Gamma_m$, which we denote as $A \cdot X^{m+1}_t$ and $A \cdot \tilde{X}^{m+1}_t(t-2^k)$. Thus, under our assumptions on $k$ and $t$, $d_{m+1}(X^{m+1}_t,\tilde{X}^{m+1}_t(t-2^k))$ has the same distribution as $d_{m}(A \cdot X^{m}_t,A \cdot \tilde{X}^{m}_t(t-2^k))$. Hence we obtain by the inductive hypothesis
$$\sum_{k=0}^{N_m} \sum_{t=a+1}^{2^{N_{m+1}}-a} \frac{\E[d_{m+1}(X^{m+1}_t,\tilde{X}^{m+1}_t(t-2^k))^p]}{2^{kp}} = \sum_{k=0}^{N_m} \sum_{t=a+1}^{2^{N_{m+1}}-a} \frac{\E[d_{m}(A \cdot X^{m}_t,A \cdot \tilde{X}^{m}_t(t-2^k))^p]}{2^{kp}}$$
$$= \sum_{k=0}^{N_m} \sum_{T=1}^A \left(\sum_{t=a+(T-1)2^{N_m}+1}^{a+T2^{N_m}}\frac{\E[d_{m}(A \cdot X^{m}_t,A \cdot \tilde{X}^{m}_t(t-2^k))^p]}{2^{kp}}\right)$$
$$= \sum_{k=0}^{N_m} \sum_{T=1}^A \sum_{t=1}^{2^{N_m}} \frac{\E[d_{m}(X^m_t,\tilde{X}^m_t(t-2^k))^p]}{2^{kp}} \overset{\text{ind hyp}}{\geq} \sum_{T=1}^A \frac{m}{8}2^{N_m}\Pi_{i=1}^{m-1}(1-(i+1)^{-2})$$
$$= 2^{N_m}A\frac{m}{8}\Pi_{i=1}^{m-1}(1-(i+1)^{-2}) = 2^{N_{m+1}}(1-2^{-\lceil 2\log_2(m+1) \rceil})\frac{m}{8}\Pi_{i=1}^{m-1}(1-(i+1)^{-2})$$
$$\geq 2^{N_{m+1}}(1-(m+1)^{-2})\frac{m}{8}\Pi_{i=1}^{m-1}(1-(i+1)^{-2}) = \frac{m}{8}2^{N_{m+1}}\Pi_{i=1}^{m}(1-(i+1)^{-2})$$

In summary,
\begin{equation} \label{eq:badconvexity1}
\sum_{k=0}^{N_m} \sum_{t=a+1}^{2^{N_{m+1}}-a} \frac{\E[d_{m+1}(X^{m+1}_t,\tilde{X}^{m+1}_t(t-2^k))^p]}{2^{kp}} \geq \frac{m}{8}2^{N_{m+1}}\Pi_{i=1}^{m}(1-(i+1)^{-2})
\end{equation}

Now consider $k$ and $t$ in the range $0 \leq k \leq N_{m+1}-1$, $1 \leq t \leq 2^k$, so that $t-2^k \leq 0$. Note that this means this range is disjoint from the one previously considered. Since $t-2^k \leq 0$, the random walks $X^{m+1}$ and $\tilde{X}^{m+1}(t-2^k)$ evolved independently immediately. Thus, with probability 1/2, $X^{m+1}$ and $\tilde{X}^{m+1}(t-2^k)$ belong to different copies of $\Gamma_{m+1}'$ in $\Gamma_{m+1}$. This implies that, with probability 1/2, $d_{m+1}(X^{m+1}_t,\tilde{X}^{m+1}_t(t-2^k)) = 2t$. Thus,
$$\sum_{k=0}^{N_{m+1}-1} \sum _{t=1}^{2^k} \frac{\E[d_{m+1}(X^{m+1}_t,\tilde{X}^{m+1}_t(t-2^k))^p]}{2^{kp}} \geq \sum_{k=0}^{N_{m+1}-1} \sum _{t=1}^{2^k} \frac{(2t)^p}{2^{kp+1}}$$
$$\overset{\text{Lem }\ref{lem:psum}}{>} \sum_{k=0}^{N_{m+1}-1} \frac{2^{k(p+1)}}{2^{kp+2}} = \sum_{k=0}^{N_{m+1}-1} 2^{k-2} = 2^{N_{m+1}-2} - \frac{1}{4} \geq \frac{1}{8}2^{N_{m+1}}$$
In summary,
\begin{equation} \label{eq:badconvexity2}
\sum_{k=0}^{N_{m+1}-1} \sum _{t=1}^{2^k} \frac{\E[d_{m+1}(X^{m+1}_t,\tilde{X}^{m+1}_t(t-2^k))^p]}{2^{kp}} > \frac{1}{8}2^{N_{m+1}}
\end{equation}

Again, notice that in \eqref{eq:badconvexity1} and \eqref{eq:badconvexity2}, the range of $t$, $k$ we consider are disjoint from each other and are subsets of the range $0 \leq k \leq N_{m+1}$, $1 \leq t \leq 2^{N_{m+1}}$. Thus, by adding \eqref{eq:badconvexity1} and \eqref{eq:badconvexity2}, we obtain
$$\sum_{k=0}^{N_{m+1}} \sum_{t=1}^{2^{N_{m+1}}} \frac{\E[d_m(X^m_t,\tilde{X}^m_t(t-2^k))^p]}{2^{kp}} > \frac{m}{8}2^{N_{m+1}}\Pi_{i=1}^{m}(1-(i+1)^{-2}) + \frac{1}{8}2^{N_{m+1}}$$
$$> \left(\frac{m+1}{8}\right)2^{N_{m+1}}\Pi_{i=1}^{m}(1-(i+1)^{-2})$$
completing the inductive step.
\end{proof}

\begin{lemma} \label{lem:badconvexity2}
$$\sum_{k=0}^{\infty} \sum_{t=1}^{2^{N_m}} \frac{\E[d_m(X^m_t,\tilde{X}^m_t(t-2^k))^p]}{2^{kp}} \gtrsim m2^{N_m}$$
for all $p > 0$.
\end{lemma}

\begin{proof}
This follows from Lemma \ref{lem:badconvexity1} and the fact that $\Pi_{i=1}^{m-1}(1-(i+1)^{-2}) > \Pi_{i=1}^{\infty}(1-(i+1)^{-2}) > 0$ for all $m \geq 0$.
\end{proof}

\subsection{Mapping the Graphs into $J^{r-1}(\R)$} \label{ss:lowerbound}
\begin{lemma} \label{lem:phidef}
There exists $\phi \in C^{r-1,1}([0,1])$ such that
\begin{enumerate}
\item \label{phidef1} $\phi$ is symmetric across the line $x = \frac{1}{2}$, that is, $\phi(x) = \phi(1-x)$ for all $x \in [0,\frac{1}{2}]$.
\item \label{phidef2} $\phi(x) \geq (2x)^r$ for all $x \in [0,\frac{1}{2}]$.
\item \label{phidef3} $[j^{r-1}(0)](\phi) = (0,0)$, and thus by (\ref{phidef1}), $[j^{r-1}(1)](\phi) = (1,0)$.
\item \label{phidef4} For every integer $0 \leq i < 2^r$ and every $x \in [i2^{-r},(i+1)2^{-r})$, $\phi^{(r)}(x) = \phi^{(r)}(i2^{-r})$ (so $\phi^{(r)}$ is constant on intervals of this form).
\end{enumerate}
Since $\phi \in C^{r-1,1}([0,1])$, $\phi^{(r)} \in L^\infty([0,1])$. We also remark here that whenever dealing with $L^\infty$ functions, we choose representatives that are everywhere (not just almost everywhere) bounded by their norm.
\end{lemma}

\begin{proof}
The proof is by induction on $r$. For the base case $r = 1$, define
$$\phi(x) := \left\{\begin{matrix} 2x & x \in [0,\frac{1}{2}] \\ 2-2x & x \in [\frac{1}{2},1] \end{matrix}\right.$$
$\phi$ satisfies (\ref{phidef1}) - (\ref{phidef4}).

Now suppose such a function $\phi$ exists for some $r \geq 1$. We'll construct a function $\psi$ that satisfies (\ref{phidef1}) - (\ref{phidef4}) for $r+1$. Define $\overline{\phi} \in C^{r-1,1}([0,1])$ by
$$\overline{\phi}(x) := \left\{\begin{matrix} \phi(2x) & x \in [0,\frac{1}{2}] \\ -\phi(2-2x) & x \in [\frac{1}{2},1] \end{matrix}\right.$$
Then define $\overline{\Phi} \in C^{r,1}([0,1])$ by
$$\overline{\Phi}(x) := \int_0^x \overline{\phi}(\xi) d\xi$$
$\overline{\Phi}$ satisfies (\ref{phidef1}), (\ref{phidef3}), and (\ref{phidef4}) by the inductive hypothesis. Note that the inductive hypothesis applied to (\ref{phidef2}) implies $\overline{\phi}(x) \geq 2^r(2x)^r$ for every $x \in [0,\frac{1}{4}]$, and hence
$$\overline{\Phi}(x) \geq \frac{2^{r-1}}{r+1}(2x)^{r+1} \geq \frac{1}{2}(2x)^{r+1}$$
Also, since $\phi \geq 0$, (which follows from the inductive hypothesis applied to (\ref{phidef1}) and (\ref{phidef2})),
$$\overline{\Phi}(x) \geq \overline{\Phi}\left(\frac{1}{4}\right) \geq \left(\frac{1}{2}\right)^{r+2}$$
for all $x \in [\frac{1}{4},\frac{1}{2}]$. Together, these two inequalities imply
$$\psi(x) := 2^{r+2}\overline{\Phi}(x) \geq (2x)^{r+1}$$
for all $x \in [0,\frac{1}{2}]$. Thus, $\psi$ satisfies (\ref{phidef1})-(\ref{phidef4}), completing the inductive step.
\end{proof}

See Figure \ref{fig:phi} for graphs of $\phi$ and its first two derivatives when $r = 3$. Note that these graphs are not on the same scale.

\begin{figure}
\includegraphics[scale=.425]{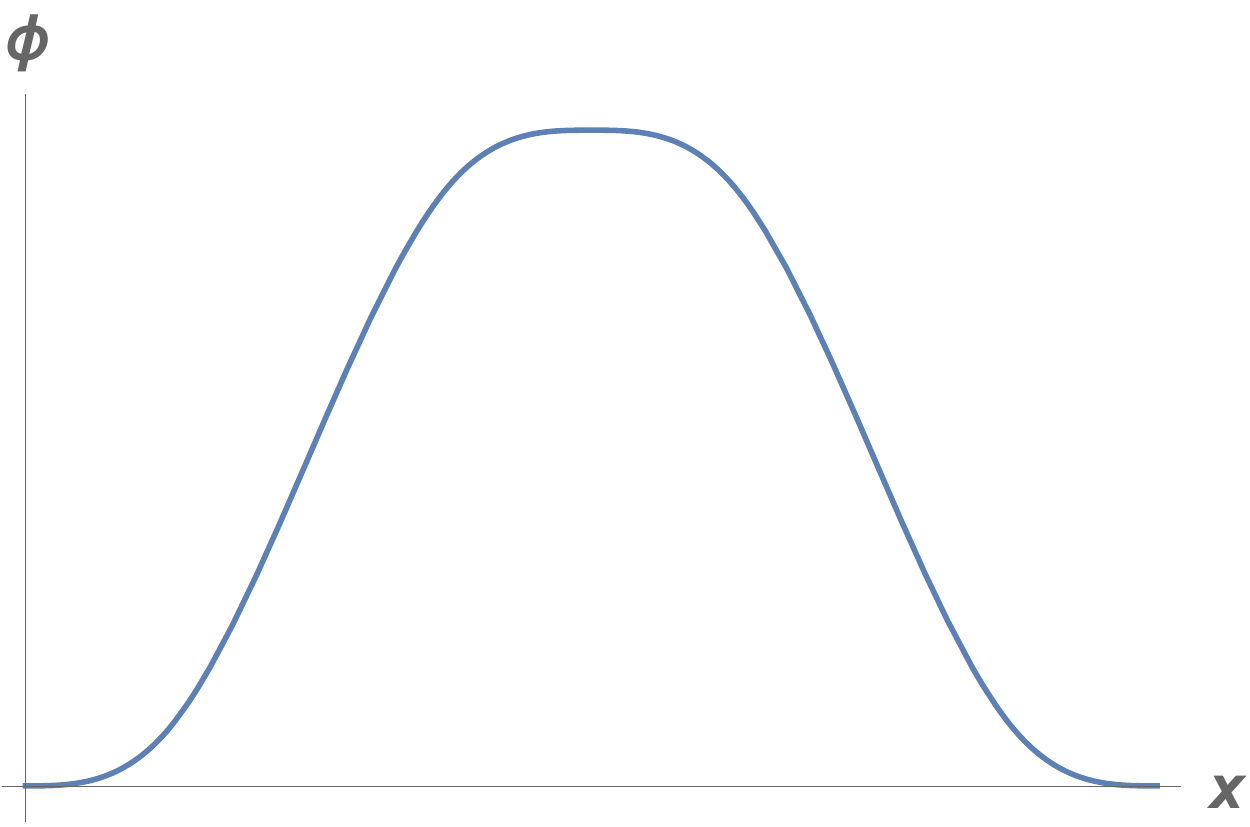}
\includegraphics[scale=.425]{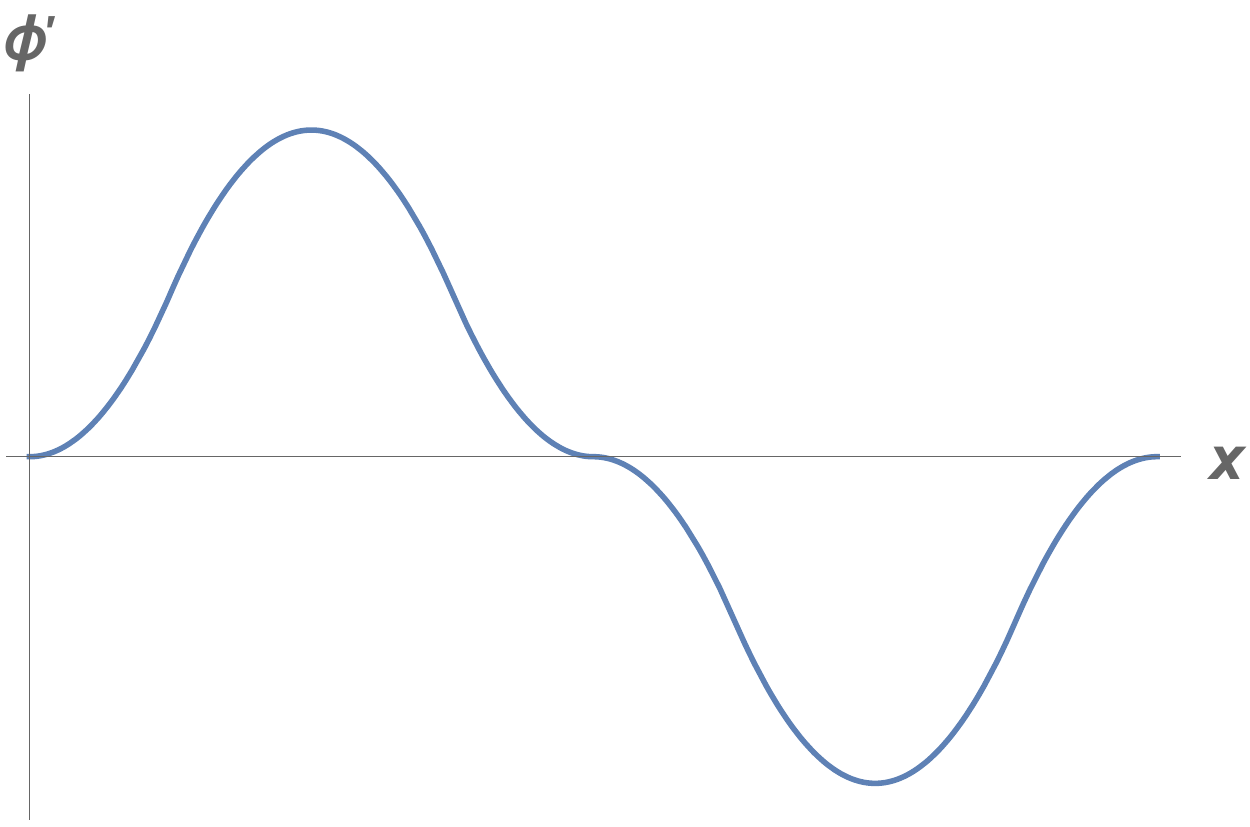}
\includegraphics[scale=.425]{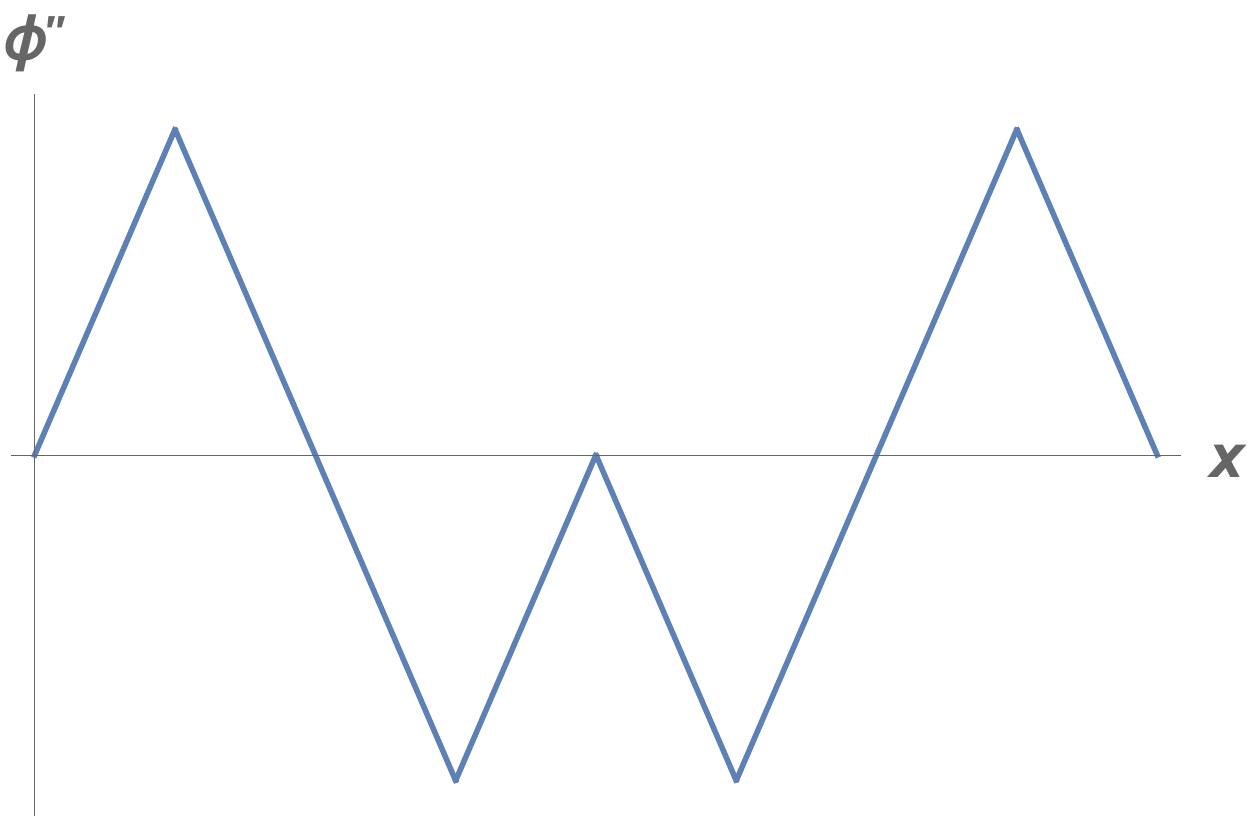}
\caption{Graphs of the function $\phi$ from Lemma \ref{lem:phidef} and its first two derivatives when $r = 3$. Note that these are not shown to the same scale.}
\label{fig:phi}
\end{figure}

\begin{lemma} \label{lem:mapintojetspace}
Let $\phi$ be the function from Lemma \ref{lem:phidef}. Set $N_0 = 0$, and for $m \geq 1$, set $N_m := \lceil Cm\log_2(m+1) \rceil$, where $C$ is a sufficiently large constant to be chosen later, so that $N_m \geq r$ and $N_{m+1} \geq \max(1,N_m + \lceil 2\log_2(m+1) \rceil)$. Then there exists a sequence of maps $F_m: \Gamma_m \to J^{r-1}(\R)$ such that, for all $m \geq 0$ and all directed paths $\gamma$ from $0_m$ to $1_m$ in $\Gamma_m$, there is a function $\phi_\gamma \in C^{r-1,1}([0,2^{N_m}])$ such that
\begin{enumerate}
\item \label{mapintojetspace1} $[j^{r-1}(0)](\phi_\gamma) = (0,0)$ and $[j^{r-1}(2^{N_m})](\phi_\gamma) = (2^{N_m},0)$.
\item \label{mapintojetspace2} After isometrically identifying $\gamma$ with $[0,2^{N_m}]$ via $q \mapsto d_m(q,0_m)$, $F_m$ restricted to $\gamma$ equals the jet of $\phi_\gamma$; $F_m(t) = [j^{r-1}(t)](\phi_\gamma)$.
\item \label{mapintojetspace3} For all vertical pairs $(q_1,q_2) \in \Gamma_m \cross \Gamma_m$,
$$\sqrt{m}\ln(m+1)|\pi_{0}(F_m(q_1)) - F_m(q_2))| \geq d_m(q_1,q_2)^r$$
\item \label{mapintojetspace4} Let $\gamma(X^m)$ denote the directed path followed by the random walk $X^m$ (so $\gamma(X^m)$ is itself a path-valued random variable). For all $y \in \R$, and $0 \leq t < 2^{N_m}$,
$$\E\left[\exp\left(y\left(\sup_{[t,t+1]} \phi^{(r)}_{\gamma(X^m)}\right)\right)\right] \leq \exp\left(\frac{y^2}{2}\left\|\phi^{(r)}\right\|_\infty^2\sum_{n=1}^{m} \frac{1}{n\ln(n+1)^2}\right)$$
and
$$\E\left[\exp\left(y\left(\inf_{[t,t+1]} \phi^{(r)}_{\gamma(X^m)}\right)\right)\right] \leq \exp\left(\frac{y^2}{2}\left\|\phi^{(r)}\right\|_\infty^2\sum_{n=1}^{m} \frac{1}{n\ln(n+1)^2}\right)$$
and thus there exists a constant $B < \infty$ (not depending on $y$, $t$, or $m$) such that
$$\E\left[\exp\left(y\left\|\phi^{(r)}_{\gamma(X^m)}\right\|_{L^\infty[t,t+1]}\right)\right] \leq 2e^{By^2}$$
\item \label{mapintojetspace5} $\|\phi^{(r)}_\gamma\|_\infty \leq 2\sqrt{m} \|\phi^{(r)}\|_\infty$.
\item \label{mapintojetspace6} $\|\pi_0 \comp F_m\|_\infty \leq 2^{r}(m+1)^{Crm+1}\|\phi\|_\infty$.
\end{enumerate}
\end{lemma}

\begin{proof}
The proof is by induction on $m$. The base case $m=0$ is easy, we simply define $F_0$ to be the jet of the 0 function on $\Gamma_0 = I$. Then (\ref{mapintojetspace1}) - (\ref{mapintojetspace6}) hold. Assume such a sequence of maps $F_0, \dots F_m$ exist for some $m \geq 0$. Set
\begin{equation} \label{eq:Kdef}
K := \|\pi_0 \comp F_m\|_\infty
\end{equation}
Since $N_{m+1} \geq C(m+1)\log_2(m+2)$, we may (and do) choose $C$ sufficiently large so that
\begin{equation} \label{eq:Nmdef}
K \overset{\text{ind hyp }(\ref{mapintojetspace6})}{\leq} \|\phi\|_\infty 2^{r}(m+1)^{Crm+1} \leq \frac{2^{r(N_{m+1}- \lceil 2\log_2(m+1) \rceil -1)-1}}{\sqrt{m+1}\ln(m+2)}
\end{equation}

Define $\tilde{\phi} \in C^{r-1,1}([0,2^{N_{m+1}}])$ by
$$\tilde{\phi}(x) := \frac{2^{rN_{m+1}}}{\sqrt{m+1}\ln(m+2)}\phi(2^{-N_{m+1}}x)$$
Note that since $N_{m+1} \geq r$, Lemma \ref{lem:phidef}(\ref{phidef4}) tells us:
\begin{equation} \label{eq:phitildeconst}
\tilde{\phi}^{(r)}(x) = \tilde{\phi}^{(r)}(i)
\end{equation}
for every integer $0 \leq i < 2^{N_m}$ and every $x \in [i,i+1)$. We also have by the chain rule
\begin{equation} \label{eq:phitilder}
\left\|\tilde{\phi}^{(r)}\right\|_\infty = \frac{\left\|\phi^{(r)}\right\|_\infty}{\sqrt{m+1}\ln(m+2)}
\end{equation}
and additionally
\begin{equation} \label{eq:phitilde0}
\left\|\tilde{\phi}\right\|_\infty \leq 2^{rN_{m+1}}\|\phi\|_\infty \leq 2^{r(C(m+1)\log_2(m+2)+1)}\|\phi\|_\infty = 2^{r}(m+2)^{Cr(m+1)}\|\phi\|_\infty
\end{equation}

We will now define the function $F_{m+1}$ on $\Gamma_{m+1} = +\Gamma_{m+1}' \cup -\Gamma_{m+1}'$. Let us first work with $+\Gamma_{m+1}'$. Let $\gamma$ be a directed path from $0_m$ to $1_m$ in $+\Gamma_{m+1}'$. Then by definition of $+\Gamma_{m+1}'$, $\gamma$ consists of $a = 2^{N_{m+1}- \lceil 2\log_2(m+1) \rceil -1}$ copies of $I$, then $A = 2^{-N_m}(2^{N_{m+1}}-2a)$ copies of different directed paths $\gamma_i$, $1\leq i \leq A$, each belonging to $\Gamma_m$ and connecting $0_m$ to $1_m$, then $a$ more copies of $I$ glued together in series. Identify $\gamma$ isometrically with $[0,2^{N_{m+1}}]$ via $q \mapsto d_{m+1}(q,0_{m+1})$. Under this identification, the first set of copies of $I$ gets identified with the subinterval $[0,a]$, each $\gamma_i$ gets identified with the subinterval $[a+(i-1)2^{N_m},a+i2^{N_m}]$, and the last set of copies of $I$ gets identified with the subinterval $[2^{N_{m+1}}-a,2^{N_{m+1}}]$. We then define
\begin{equation} \label{eq:Fdef1}
\phi_\gamma := \tilde{\phi} + f_\gamma
\end{equation}
where $f_\gamma$ is defined as follows: $f_\gamma$ is identically $0$ on $[0,a] \cup [2^{N_{m+1}}-a,2^{N_{m+1}}]$, and $f_\gamma(x) = \phi_{\gamma_i}(x-a-(i-1)2^{N_m})$ on $[a+(i-1)2^{N_m},a+i2^{N_m}]$ ($\phi_{\gamma_i}$ is given to us by the inductive hypothesis). By the inductive hypothesis applied to (\ref{mapintojetspace1}) and Lemma \ref{lem:phidef}(\ref{phidef3}), $\phi_\gamma \in C^{r-1,1}([0,2^{N_{m+1}}])$ and satisfies (\ref{mapintojetspace1}). It is also clear from this definition, \eqref{eq:phitilder}, and the inductive hypothesis applied to (\ref{mapintojetspace5}) that
$$\left\|\phi_\gamma^{(r)}\right\|_\infty \overset{\eqref{eq:Fdef1}}{\leq} \left\|\tilde{\phi}^{(r)}\right\|_\infty + \max_{1 \leq i \leq A}\left\|\phi_{\gamma_i}^{(r)}\right\|_\infty \overset{\eqref{eq:phitilder}}{\leq} \frac{\left\|\phi^{(r)}\right\|_\infty}{\sqrt{m+1}\ln(m+2)} + \max_{1 \leq i \leq A}\left\|\phi_{\gamma_i}^{(r)}\right\|_\infty$$
$$\overset{\text{ind hyp }(\ref{mapintojetspace5})}{\leq} \frac{\left\|\phi^{(r)}\right\|_\infty}{\sqrt{m+1}\ln(m+2)} + 2\sqrt{m}\left\|\phi^{(r)}\right\|_\infty \leq 2\sqrt{m+1}\left\|\phi^{(r)}\right\|_\infty$$
verifying (\ref{mapintojetspace5}). We can finally define $F_{m+1}$ on $+\Gamma_{m+1}'$ by declaring it to be the jet of $\phi_\gamma$ on $\gamma$. We need to check that $F_{m+1}$ is well-defined. Since every point of $+\Gamma_{m+1}$ is contained in some directed path from $0_m$ to $1_m$, we only need to check what happens when one point belongs to two different paths. Let $q \in +\Gamma_{m+1}'$ and suppose $q \in \gamma \cap \gamma'$ for some directed paths $\gamma,\gamma'$ from $0_{m+1}$ to $1_{m+1}$ in $+\Gamma_{m+1}'$. Set $t := d(q,0_{m+1})$. There are two cases: $t \in [0,a] \cup [2^{N_{m+1}}-a,2^{N_{m+1}}]$ or $t \in [a+(i-1)2^{N_m},a+i2^{N_m}]$ for some $i$. Assume the first case holds. Then our definition of $F_{m+1}(q)$ based on either $q \in \gamma$ or $q \in \gamma'$ is
$$F_{m+1}(q) = [j^{r-1}(t)](\tilde{\phi})$$
so well-definedness holds in this case. In the other case, our definition of $F_{m+1}(q)$ based on $q \in \gamma$ is, by the inductive hypothesis applied to (\ref{mapintojetspace2}),
$$F_{m+1}(q) = [j^{r-1}(t)](\tilde{\phi}) + ([j^{r-1}(t-a-(i-1)2^{N_m})](\phi_{\gamma_i}) + (a+(i-1)2^{N_m}-t,0)$$
$$\overset{\text{ind hyp}}{=} [j^{r-1}(t)](\tilde{\phi}) + F_m(q) + (a+(i-1)2^{N_m}-t,0)$$
and likewise based on $q \in \gamma'$,
$$F_{m+1}(q) = [j^{r-1}(t)](\tilde{\phi}) + ([j^{r-1}(t-a-(i-1)2^{N_m})](\phi_{\gamma'_i}) + (a+(i-1)2^{N_m}-t,0)$$
$$\overset{\text{ind hyp}}{=} [j^{r-1}(t)](\tilde{\phi}) + F_m(q) + (a+(i-1)2^{N_m}-t,0)$$ 
(note that the term $(a+(i-1)2^{N_m}-t,0)$ is present so that the $x$-coordinate of the entire expression will be $t$, and that we identify $q$ as belonging to a copy of $\Gamma_m$ so that $F_m(q)$ makes sense) so well-definedness holds in this case as well. Thus $F_{m+1}$ is well-defined on $+\Gamma_{m+1}'$. We define $F_{m+1}$ on $-\Gamma_{m+1}$ by $F_{m+1}(q) = -F_{m+1}(\iota(q))$, where $\iota: +\Gamma_{m+1}' \to -\Gamma_{m+1}'$ is the involution. It follows from this that if $\gamma$ is a directed $0_{m+1}$-$1_{m+1}$ path in $-\Gamma_{m+1}'$, then $\phi_\gamma = -\phi_{\iota(\gamma)}$. Thus, (\ref{mapintojetspace1}) and (\ref{mapintojetspace2}) are satisfied. It remains to show (\ref{mapintojetspace3}), (\ref{mapintojetspace4}), and (\ref{mapintojetspace6}). Before doing so, let us summarize the discussion on $F_{m+1}$ of this paragraph: for $q \in \Gamma_{m+1}$ and $t = d_{m+1}(q,0_{m+1})$,
\begin{equation} \label{eq:Fdef2}
\begin{multlined}
F_{m+1}(q) = \left\{\begin{matrix}
[j^{r-1}(t)](\tilde{\phi}) & t \in [0,a] \cup [2^{N_{m+1}}-a,2^{N_{m+1}}] \\ & q \in +\Gamma_{m+1}' \\ \\
[j^{r-1}(t)](\tilde{\phi}) \\ + F_m(q) + (a+(i-1)2^{N_m}-t,0) & t \in [a+(i-1)2^{N_m},a+i2^{N_m}] \\ &  q \in +\Gamma_{m+1}' \\ \\
[j^{r-1}(t)](-\tilde{\phi}) & t \in [0,a] \cup [2^{N_{m+1}}-a,2^{N_{m+1}}] \\ & q \in -\Gamma_{m+1}' \\ \\
[j^{r-1}(t)](-\tilde{\phi}) \\ - F_m(q) - (a+(i-1)2^{N_m}-t,0) & t \in [a+(i-1)2^{N_m},a+i2^{N_m}] \\ &  q \in -\Gamma_{m+1}' \end{matrix}\right.
\end{multlined}
\end{equation}

\begin{figure}
\includegraphics[scale=.64]{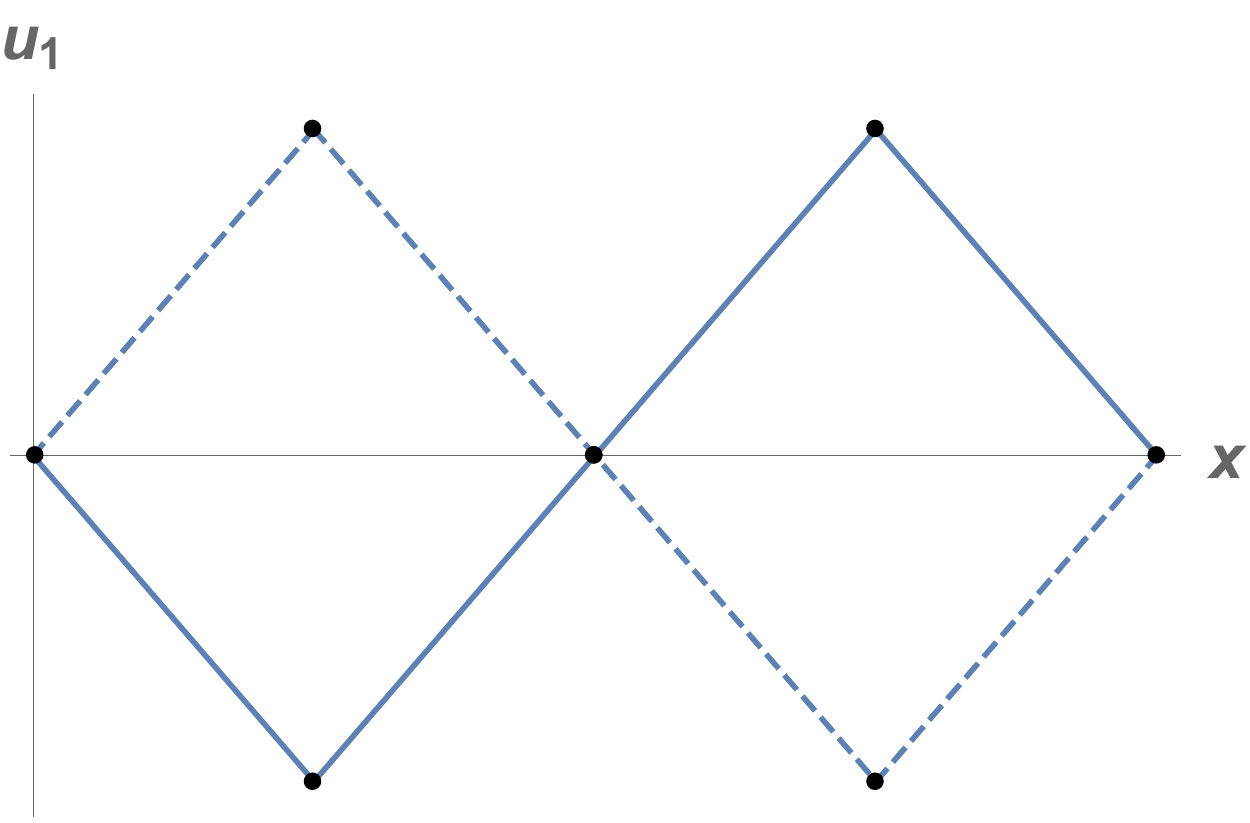}
\includegraphics[scale=.64]{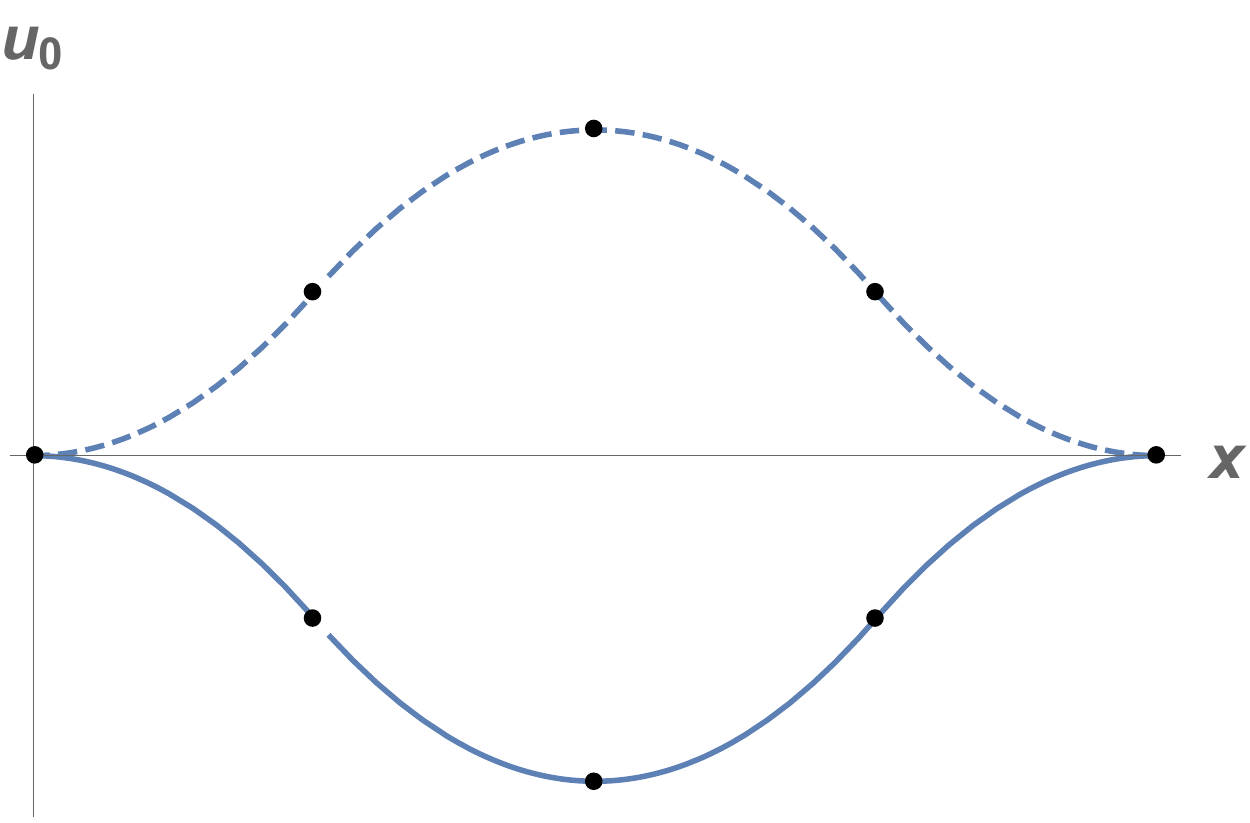}
\includegraphics[scale=.64]{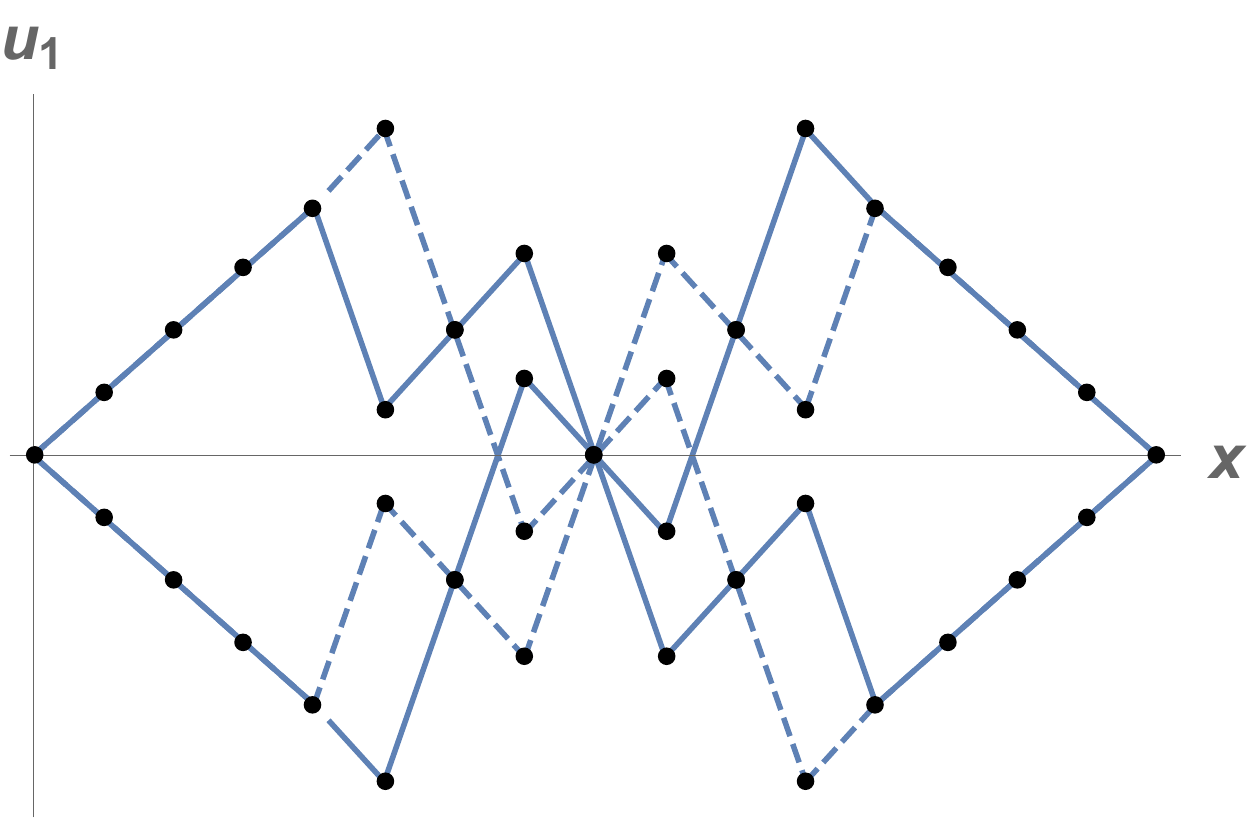}
\includegraphics[scale=.64]{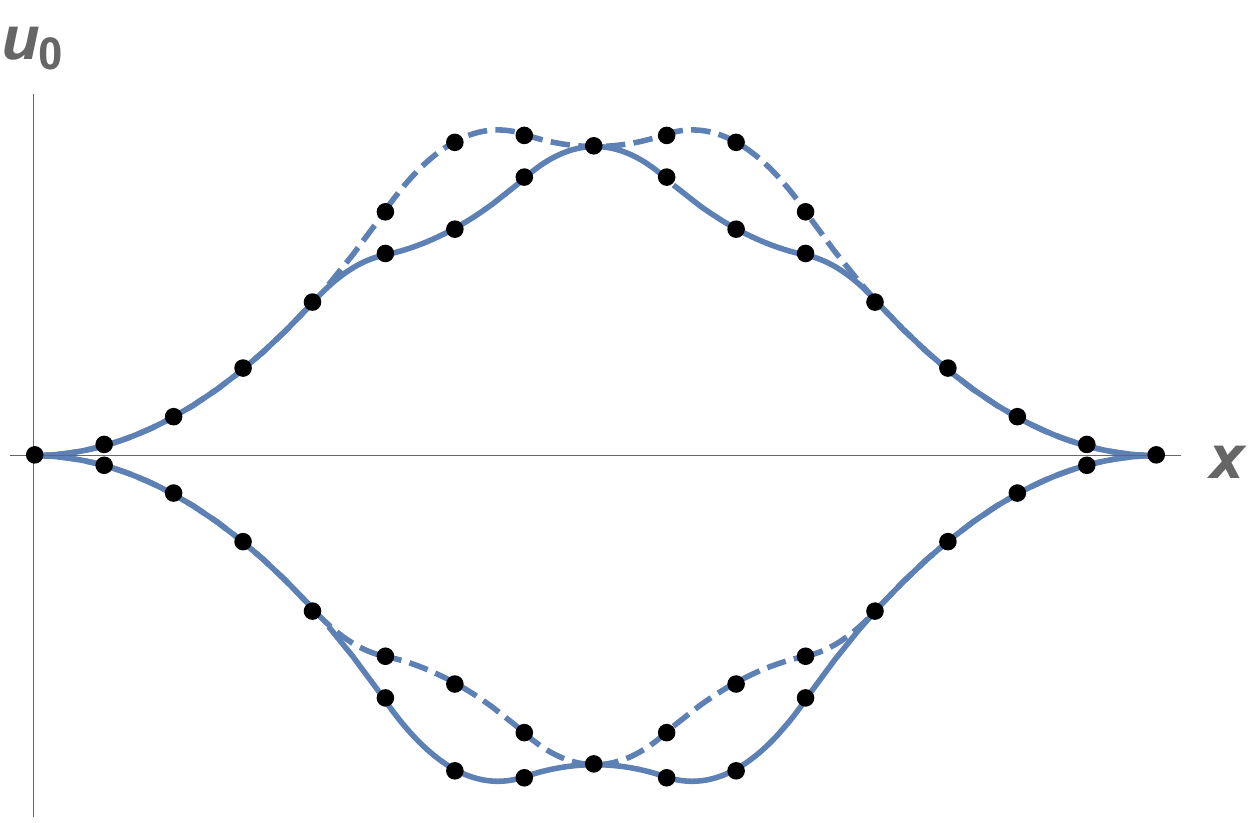}
\caption{Above, the image of $\Gamma_1$, and below, the image of $\Gamma_2$, based on $N_0 = 0$, $N_1 = 2$, $N_2 = 4$, in $J^1(\R)$ under the map $F_2$. $J^1(\R)$ is identified with $\R^3$ via the coordinates $x,u_1,u_0$. These are not drawn to the same scale. The two images on the right are respectively graph isomorphic to $\Gamma_1$ and $\Gamma_2$.}
\label{fig:Fm}
\end{figure}

See Figure \ref{fig:Fm} for the images of $\Gamma_1$ and $\Gamma_2$, based on $N_0 = 0$, $N_1 = 2$, $N_2 = 4$, in $J^1(\R)$. Using \eqref{eq:Fdef2}, we can quickly verify (\ref{mapintojetspace6}):
$$\|\pi_0 \comp F_{m+1}\|_\infty \overset{\eqref{eq:Fdef2}}{\leq} \left\|\tilde{\phi}\right\|_\infty + \|\pi_0 \comp F_{m}\|_\infty \overset{\text{ind hyp }(\ref{mapintojetspace6})}{\leq} \left\|\tilde{\phi}\right\|_\infty + 2^{r}(m+1)^{Crm+1}\|\phi\|_\infty$$
$$\overset{\eqref{eq:phitilde0}}{\leq} 2^{r}(m+2)^{Cr(m+1)}\|\phi\|_\infty + 2^{r}(m+1)^{Crm+1}\|\phi\|_\infty \leq 2^{r}(m+2)^{Cr(m+1)+1}\|\phi\|_\infty$$
(\ref{mapintojetspace3}) and (\ref{mapintojetspace4}) require more involved arguments.

\underline{Proof of (\ref{mapintojetspace3})}. Let $(q_1,q_2) \in \Gamma_{m+1} \cross \Gamma_{m+1}$ be a vertical pair. By definition of vertical pair, $d_{m+1}(q_1,0_{m+1}) = d_{m+1}(q_2,0_{m+1})$. Let $t$ denote this common value. There are two cases, $q_1,q_2$ belong to the same copy of $\Gamma_{m+1}'$, or they belong to different copies. First assume they belong to the same copy. Without loss of generality say $+\Gamma_{m+1}'$. Then there are two subcases for $t$: $t \in [0,a] \cup [2^{N_{m+1}}-a,2^{N_{m+1}}]$ or $t \in [a+(i-1)2^{N_m},a+i2^{N_m}]$ for some $1 \leq i \leq A$. Assume the first subcase holds. Then by construction of $+\Gamma_{m+1}'$, $q_1$, $q_2$ belong to a copy of $I$, and thus the equality $d_{m+1}(q_1,0_{m+1}) = d_{m+1}(q_2,0_{m+1})$ implies $q_1 = q_2$, so (\ref{mapintojetspace3}) trivially holds. Assume the second subcase for $t$. Then
$$|\pi_0(F_{m+1}(q_1) - F_{m+1}(q_2))| \overset{\eqref{eq:Fdef2}}{=} |\pi_0(F_{m}(q_1) - F_{m}(q_2))|$$
and so (\ref{mapintojetspace3}) holds by the inductive hypothesis.

Now assume we are in the second case where $q_1,q_2$ belong to different copies of $\Gamma_{m+1}'$. Without loss of generality, assume $q_1 \in +\Gamma_{m+1}'$ and $q_2 \in -\Gamma_{m+1}'$. Observe that under this assumption, $d_{m+1}(q_1,q_2) = 2t$ if $t \leq 2^{N_{m+1}-1}$ and $d_{m+1}(q_1,q_2) = 2(2^{N_{m+1}}-t)$ if $t \geq 2^{N_{m+1}-1}$. Because of the symmetry of $\tilde{\phi}$ about the line $x = 2^{N_{m+1}-1}$, it suffices to only check the case $t \leq 2^{N_{m+1}-1}$. Let us first record the following inequality:
\begin{equation} \label{eq:phitilde0geq}
\pi_{0}([j^{r-1}(t)](\tilde{\phi})) \geq  \frac{(2t)^r}{\sqrt{m+1}\ln(m+2)}
\end{equation}
which can be proven by
$$\pi_{0}([j^{r-1}(t)](\tilde{\phi})) = \tilde{\phi}(t) = \frac{2^{rN_{m+1}}}{\sqrt{m+1}\ln(m+2)}\phi(2^{-N_{m+1}}t) \overset{\text{Lem }\ref{lem:phidef}(\ref{phidef2})}{\geq} \frac{(2t)^r}{\sqrt{m+1}\ln(m+2)}$$
Again split into two subcases: $t \in [0,a]$ or $t \in [a,2^{N_{m+1}-1}]$. In the first subcase we have
$$\pi_{0}(F_{m+1}(q_1)) \overset{\eqref{eq:Fdef2}}{=} \pi_{0}([j^{r-1}(t)](\tilde{\phi})) \overset{\eqref{eq:phitilde0geq}}{\geq} \frac{(2t)^r}{\sqrt{m+1}\ln(m+2)}$$
and
$$\pi_{0}(F_{m+1}(q_2)) \overset{\eqref{eq:Fdef2}}{=} \pi_{0}([j^{r-1}(t)](-\tilde{\phi})) \overset{\eqref{eq:phitilde0geq}}{\leq} -\frac{(2t)^r}{\sqrt{m+1}\ln(m+2)}$$
and thus
$$|\pi_{0}(F_{m+1}(q_1) - F_{m+1}(q_2))| \geq \frac{2(2t)^r}{\sqrt{m+1}\ln(m+2)} = \frac{2d_{m+1}(q_1,q_2)^r}{\sqrt{m+1}\ln(m+2)}$$
proving (\ref{mapintojetspace3}) in this subcase.

Now assume the second subcase, $t \in [a,2^{N_{m+1}-1}]$. Then
$$\pi_{0}(F_{m+1}(q_1)) \overset{\eqref{eq:Fdef2}}{=} \pi_{0}([j^{r-1}(t)](\tilde{\phi}) + F_m(q_1) + (a+(i-1)2^{N_m}-t,0)) = \pi_{0}([j^{r-1}(t)](\tilde{\phi})) + \pi_{0}(F_m(q_1))$$
$$\overset{\eqref{eq:Kdef}}{\geq} \pi_{0}([j^{r-1}(t)](\tilde{\phi})) - K \overset{\eqref{eq:Nmdef}}{\geq} \pi_{0}([j^{r-1}(t)](\tilde{\phi})) - \frac{2^{r(N_{m+1}- \lceil 2\log_2(m+1) \rceil)-1}}{\sqrt{m+1}\ln(m+2)}$$
$$\overset{\eqref{eq:phitilde0geq}}{\geq} \frac{(2t)^r - 2^{r(N_{m+1}- \lceil 2\log_2(m+1) \rceil)-1}}{\sqrt{m+1}\ln(m+2)} = \frac{(2t)^r - (2a)^r/2}{\sqrt{m+1}\ln(m+2)} \geq \frac{(2t)^r - (2t)^r/2}{\sqrt{m+1}\ln(m+2)} = \frac{(2t)^r}{2\sqrt{m+1}\ln(m+2)}$$
Similarly,
$$\pi_{0}(F_{m+1}(q_2)) \leq -\frac{(2t)^r}{2\sqrt{m+1}\ln(m+2)}$$
and thus
$$|\pi_{0}(F_{m+1}(q_1) - F_{m+1}(q_2))| \geq \frac{(2t)^r}{\sqrt{m+1}\ln(m+2)} = \frac{d_{m+1}(q_1,q_2)^r}{\sqrt{m+1}\ln(m+2)}$$
proving (\ref{mapintojetspace3}) in this final subcase.

\underline{Proof of (\ref{mapintojetspace4})}. Let $0 \leq t < 2^{N_{m+1}}$ be an arbitrary integer. Again we consider two cases for $t$: $t \in [0,a) \cup [2^{N_{m+1}}-a,2^{N_{m+1}})$ or $t \in [a,2^{N_{m+1}}-a)$. Assume the first case holds. There are two subcases to consider for $\gamma(X^{m+1})$: $\gamma(X^{m+1})$ belongs to $+\Gamma_{m+1}'$ or $\gamma(X^{m+1})$ belongs to $-\Gamma_{m+1}'$. These are complementary events each occuring with probability 1/2. Restricted to the first event, for every $x \in [t,t+1]$,
$$\phi^{(r)}_{\gamma(X^{m+1})}(x) \overset{\eqref{eq:Fdef1}}{=} \tilde{\phi}^{(r)}(x) + f_{\gamma(X^{m+1})}(x) = \tilde{\phi}^{(r)}(x) \overset{\eqref{eq:phitildeconst}}{=} \tilde{\phi}^{(r)}(t) $$
where the second equality holds by the definition of $f$ succeeding \eqref{eq:Fdef1}.
Thus,
$$\sup_{[t,t+1]} \phi^{(r)}_{\gamma(X^{m+1})} = \inf_{[t,t+1]} \phi^{(r)}_{\gamma(X^{m+1})} = \tilde{\phi}^{(r)}(t)$$
Likewise, for the second subcase where we restrict to the event that $\gamma(X^{m+1})$ belongs to $-\Gamma_{m+1}'$,
$$\sup_{[t,t+1]} \phi^{(r)}_{\gamma(X^{m+1})} = \inf_{[t,t+1]} \phi^{(r)}_{\gamma(X^{m+1})} = -\tilde{\phi}^{(r)}(t)$$
Combining these yields
$$\E\left[\exp\left(y\left(\sup_{[t,t+1]} \phi^{(r)}_{\gamma(X^{m+1})}\right)\right)\right] = \frac{1}{2}\left(\exp\left(y\tilde{\phi}^{(r)}(t)\right) + \exp\left(-y\tilde{\phi}^{(r)}(t)\right)\right)$$
$$= \cosh\left(y\tilde{\phi}^{(r)}(t)\right) \leq \cosh\left(y\left\|\tilde{\phi}^{(r)}\right\|_\infty\right) \overset{\eqref{eq:phitilder}}{=} \cosh\left(y\left\|\phi^{(r)}\right\|_\infty\frac{1}{\sqrt{m+1}\ln(m+2)}\right)$$
$$\overset{\text{Lem }\ref{lem:coshexp}}{\leq} \exp\left(\frac{y^2}{2}\left\|\phi^{(r)}\right\|_\infty^2\frac{1}{(m+1)\ln(m+2)^2}\right)$$
and the same estimate holds for the essential infimum, verifying (\ref{mapintojetspace4}) in this case.

Now consider the second case, $t \in [a+(i-1)2^{N_m},a+i2^{N_m}]$ for some $1 \leq i \leq A$. Again, there are two subcases to consider for $\gamma(X^{m+1})$: $\gamma(X^{m+1})$ belongs to $+\Gamma_{m+1}'$ or $\gamma(X^{m+1})$ belongs to $-\Gamma_{m+1}'$. Restricted to the first event, and for the range of $t$ under consideration, $X^{m+1}$ is equal in distribution to a copy of $X^m$ (after an appropriate shift in the time parameter), by definition of $+\Gamma_{m+1}'$. Thus, for every $x \in [t,t+1]$,
$$\phi^{(r)}_{\gamma(X^{m+1})}(x) \overset{\eqref{eq:Fdef1}}{=} \tilde{\phi}^{(r)}(x) + f_{\gamma(X^{m+1})}(x) = \tilde{\phi}^{(r)}(x) + \phi_{\gamma(X^m)}(x') \overset{\eqref{eq:phitildeconst}}{=} \tilde{\phi}^{(r)}(t) + \phi_{\gamma(X^m)}(x')$$
where $x' = x-a-(i-1)2^{N_m}$, and the second equality holds by the definition of $f$ succeeding \eqref{eq:Fdef1}.
Thus,
$$\sup_{[t,t+1]} \phi^{(r)}_{\gamma(X^{m+1})} = \tilde{\phi}^{(r)}(t) + \sup_{[t',t'+1]} \phi_{\gamma(X^m)}$$
$$\inf_{[t,t+1]} \phi^{(r)}_{\gamma(X^{m+1})} = \tilde{\phi}^{(r)}(t) + \inf_{[t',t'+1]} \phi_{\gamma(X^m)}$$
where $t' =  t-a-(i-1)2^{N_m}$.
Likewise, for the second subcase where we restrict to the event that $\gamma(X^{m+1})$ belongs to $-\Gamma_{m+1}'$,
$$\sup_{[t,t+1]} \phi^{(r)}_{\gamma(X^{m+1})} = -\tilde{\phi}^{(r)}(t) - \inf_{[t',t'+1]} \phi_{\gamma(X^m)}$$
$$\inf_{[t,t+1]} \phi^{(r)}_{\gamma(X^{m+1})} = -\tilde{\phi}^{(r)}(t) - \sup_{[t',t'+1]} \phi_{\gamma(X^m)}$$

Combining these and using the inductive hypothesis applied to (\ref{mapintojetspace4}) and some basic monotonicity and symmetry properties of $\cosh$ yields
$$\E\left[\exp\left(y\left(\sup_{[t,t+1]} \phi^{(r)}_{\gamma(X^{m+1})}\right)\right)\right]$$
$$= \frac{1}{2}\exp\left(y\tilde{\phi}^{(r)}(t)\right)\E\left[\exp\left(y\left(\sup_{[t',t'+1]} \phi^{(r)}_{\gamma(X^{m+1})}\right)\right)\right]$$
$$+ \frac{1}{2}\exp\left(-y\tilde{\phi}^{(r)}(t)\right)\E\left[\exp\left(-y\left(\inf_{[t',t'+1]} \phi^{(r)}_{\gamma(X^{m+1})}\right)\right)\right]$$
$$\overset{\text{ind hyp}}{\leq} \frac{1}{2}\exp\left(y\tilde{\phi}^{(r)}(t)\right)\exp\left(\frac{y^2}{2}\left\|\phi^{(r)}\right\|_\infty^2\sum_{n=1}^{m} \frac{1}{n\ln(n+1)^2}\right)$$
$$+ \frac{1}{2}\exp\left(-y\tilde{\phi}^{(r)}(t)\right)\exp\left(\frac{(-y)^2}{2}\left\|\phi^{(r)}\right\|_\infty^2\sum_{n=1}^{m} \frac{1}{n\ln(n+1)^2}\right)$$
$$= \cosh\left(y\tilde{\phi}^{(r)}(t)\right)\exp\left(\frac{y^2}{2}\left\|\phi^{(r)}\right\|_\infty^2\sum_{n=1}^{m} \frac{1}{n\ln(n+1)^2}\right)$$
$$ \leq \cosh\left(y\left\|\tilde{\phi}^{(r)}\right\|_\infty\right)\exp\left(\frac{y^2}{2}\left\|\phi^{(r)}\right\|_\infty^2\sum_{n=1}^{m} \frac{1}{n\ln(n+1)^2}\right)$$
$$\overset{\eqref{eq:phitilder}}{=} \cosh\left(y\left\|\phi^{(r)}\right\|_\infty\frac{1}{\sqrt{m+1}\ln(m+2)}\right)\exp\left(\frac{y^2}{2}\left\|\phi^{(r)}\right\|_\infty^2\sum_{n=1}^{m} \frac{1}{n\ln(n+1)^2}\right)$$
$$\overset{\text{Lem }\ref{lem:coshexp}}{\leq} \exp\left(\frac{y^2}{2}\left\|\phi^{(r)}\right\|_\infty^2\frac{1}{(m+1)\ln(m+2)^2}\right)\exp\left(\frac{y^2}{2}\left\|\phi^{(r)}\right\|_\infty^2\sum_{n=1}^{m} \frac{1}{n\ln(n+1)^2}\right)$$
$$= \exp\left(\frac{y^2}{2}\left\|\phi^{(r)}\right\|_\infty^2\sum_{n=1}^{m+1} \frac{1}{n\ln(n+1)^2}\right)$$
and the same estimate holds for the infimum, verifying (\ref{mapintojetspace4}) in this case. This completes the inductive step and the proof of the lemma.
\end{proof}

\begin{theorem} \label{thm:lowerboundmain}
For every $p > 0$, $r \geq 1$, coarsely dense set $N \sbs J^{r-1}(\R)$, and $R \geq 3$, let $B_N(R) := \{x \in N: d_{CC}(0,x) \leq R\}$. Then
$$\Pi_p(B_N(R)) \gtrsim \frac{\ln(R)^{\frac{1}{p}-\frac{1}{2r}}}{\ln(\ln(R))^{\frac{1}{p}+\frac{1}{2r}}}$$
where the implicit constant can depend on $r,p$ but not on $N,R$.
\end{theorem}

\begin{proof}
Let $p,r,N$ be as above. Since the Markov convexity constant $\Pi_p$ is scale-invariant, then by applying a dilation we may assume without loss of generality that every point of $J^{r-1}(\R)$ is at a distance of at most 1 away from a point of $N$. Let $F_m: \Gamma_m \to J^{r-1}(\R)$ be the sequence of maps from Lemma \ref{lem:mapintojetspace}. Extend the domain of $t$ for the random walks on $\Gamma_m$ by $X^m_t := X^m_0$ if $t \leq 0$, and $X^m_t := X^m_{2^{N_m}}$ if $t \geq 2^{N_m}$. Each $\{X^m_t\}_{t \in \Z}$ is a Markov process on the state space $\Gamma_m$. 

With full probability, $d_{CC}(X^m_t,0_m) = \min(\max(0,t),2^{N_m})$. Since $\tilde{X}^m_t(t-2^k)$ equals $X^m_t$ in distribution, $(X^m_t,\tilde{X}^m_t(t-2^k))$ is a vertical pair with full probability. Then Lemma \ref{lem:mapintojetspace}(\ref{mapintojetspace3}) applies, and we get the following lower bound for the left hand side of the Markov convexity inequality in Definition \ref{def:Markovdef}:
$$\sum_{k=0}^{\infty} \sum_{t \in \Z} \frac{\E[d_{CC}(F_m(X^m_t),F_m(\tilde{X}^m_t(t-2^k)))^p]}{2^{kp}} \overset{\text{Lem }\ref{lem:dcclowerbound}}{\geq} \sum_{k=0}^{\infty} \sum_{t \in \Z} \frac{\E[|\pi_0(f_m(X^m_t)-f_m(\tilde{X}^m_t(t-2^k)))|^{p/r}]}{2^{kp}}$$
$$\overset{\text{Lem }\ref{lem:mapintojetspace}(\ref{mapintojetspace3})}{\geq} \frac{m^{-\frac{p}{2r}}}{\ln(m+1)^{\frac{p}{s}}}\sum_{k=0}^{\infty} \sum_{t \in \Z} \frac{\E[d_m(X^m_t,\tilde{X}^m_t(t-2^k))^p]}{2^{kp}} \overset{\text{Lem }\ref{lem:badconvexity2}}{\gtrsim} m^{-\frac{p}{2r}}\ln(m+1)^{-\frac{p}{r}}m2^{N_m} = \frac{m^{1-\frac{p}{2r}}2^{N_m}}{\ln(m+1)^{\frac{p}{r}}}$$
In summary,
\begin{equation} \label{eq:MarkovIneqLHS}
\sum_{k=0}^{\infty} \sum_{t \in \Z} \frac{\E[d_{CC}(F_m(X^m_t),F_m(\tilde{X}^m_t(t-2^k)))^p]}{2^{kp}} \gtrsim \frac{m^{1-\frac{p}{2r}}2^{N_m}}{\ln(m+1)^{\frac{p}{r}}}
\end{equation}

Now we upper bound the right hand side of the Markov convexity inequality. Since \\ $d_{CC}(F_m(X^m_{t+1}),F_m(X^m_t)) = 0$ whenever $t \leq 0$ or $t \geq 2^{N_m}$,
\begin{equation} \label{eq:MarkovIneqRHSfinitesum}
\sum_{t \in \Z} \E[d_{CC}(F_m(X^m_{t+1}),F_m(X^m_t))^p] = \sum_{t=0}^{2^{N_m}-1} \E[d_{CC}(F_m(X^m_{t+1}),F_m(X^m_t))^p] =: (*)
\end{equation}
Then
$$(*) \overset{\text{Lem }\ref{lem:mapintojetspace}(\ref{mapintojetspace2})}{=} \sum_{t=0}^{2^{N_m}-1} \E\left[d_{CC}([j^{r-1}(t+1)](\phi_{\gamma(X^m)})([j^{r-1}(t)](\phi_{\gamma(X^m)}))^p\right]$$
$$\overset{\text{Lem }\ref{lem:dccupperbound}}{\leq} \sum_{t=0}^{2^{N_m}-1} \E\left[\left(1+\left\|\phi^{(r)}_{\gamma(X^m)}\right\|_{L^\infty[t,t+1]}\right)^p\right]$$
$$\lesssim \sum_{t=0}^{2^{N_m}-1} 1 + \E\left[\left\|\phi^{(r)}_{\gamma(X^m)}\right\|_{L^\infty[t,t+1]}^p\right] \overset{\text{Lems }\ref{lem:subgaussian},\ref{lem:mapintojetspace}(\ref{mapintojetspace4})}{\lesssim} \sum_{t=0}^{2^{N_m}-1} 1 = 2^{N_m}$$
In summary,
\begin{equation} \label{eq:MarkovIneqRHS}
\sum_{t \in \Z} \E[d_{CC}(F_m(X^m_{t+1}),F_m(X^m_t))^p] \lesssim 2^{N_m}
\end{equation}

Let $\pi_N: J^{r-1}(\R) \to N$ be any map so that
\begin{equation} \label{eq:piNdef}
d_{CC}(x,\pi_N(x)) \leq 1
\end{equation}
which exists by our initial assumption. We'll use $\pi_N$ to transfer inequalities \eqref{eq:MarkovIneqLHS} and \eqref{eq:MarkovIneqRHS} to corresponding inequalities on $N$. Consider the maps $\bar{F}_m: \Gamma_m \to N$ defined by $\bar{F}_m := \pi_N \comp \delta_{2m} \comp F_m$. By Lemma \ref{lem:mapintojetspace}(\ref{mapintojetspace3}), 
$$d_{CC}(\delta_{2m}(F_m(q_1)),\delta_{2m}(F_m(q_2))) \geq 2d_m(q_1,q_2) \geq 4$$
for any vertical pair $(q_1,q_2) \in \Gamma_m \cross \Gamma_m$. Combining this with \eqref{eq:piNdef} yields
$$d_{CC}(\bar{F}_m(q_1),\bar{F}_m(q_2)) \overset{\eqref{eq:piNdef}}{\geq} d_{CC}(\delta_{2m}(F_m(q_1)),\delta_{2m}(F_m(q_2))) - 2$$
$$\geq \frac{1}{2}d_{CC}(\delta_{2m}(F_m(q_1)),\delta_{2m}(F_m(q_2))) = md_{CC}(F_m(q_1),F_m(q_2))$$
for any vertical pair $(q_1,q_2)$. Combining this with \eqref{eq:MarkovIneqLHS} yields
\begin{equation} \label{eq:MarkovIneqLHS2}
\sum_{k=0}^{\infty} \sum_{t \in \Z} \frac{\E[d_{CC}(\bar{F}_m(X^m_t),\bar{F}_m(\tilde{X}^m_t(t-2^k)))^p]}{2^{kp}} \gtrsim \frac{m^{p+1-\frac{p}{2r}}2^{N_m}}{\ln(m+1)^{\frac{p}{r}}}
\end{equation}
Next,
$$d_{CC}(\bar{F}_m(X^m_{t+1}),\bar{F}_m(X^m_t)) \overset{\eqref{eq:piNdef}}{\leq} d_{CC}(\delta_{2m}(F_m(X^m_{t+1})),\delta_{2m}(F_m(X^m_t))) + 2$$
$$= 2md_{CC}(F_m(X^m_{t+1}),F_m(X^m_t)) + 2$$
Combining this with \eqref{eq:MarkovIneqRHS} and \eqref{eq:MarkovIneqRHSfinitesum} yields
\begin{equation} \label{eq:MarkovIneqRHS2}
\sum_{t \in \Z} \E[d_{CC}(\bar{F}_m(X^m_{t+1}),\bar{F}_m(X^m_t))^p] \lesssim m^p2^{N_m}
\end{equation}
For each $R \geq 1$, let $m(R)$ denote the largest $m$ so that $\bar{F}_{m(R)}(\Gamma_{m(R)}) \sbs B_N(R)$. Then \eqref{eq:MarkovIneqLHS2} and \eqref{eq:MarkovIneqRHS2} imply
\begin{equation} \label{eq:MarkovConstantbnd}
\Pi_p(B_N(R)) \gtrsim \frac{m(R)^{\frac{1}{p}-\frac{1}{2r}}}{\ln(m(R)+1)^{\frac{1}{r}}}
\end{equation}

Now we wish to estimate the quantity $m(R)$. Let $m \geq 0$ be arbitrary. Since any two points of $\Gamma_{m}$ are connected by a geodesic that is a piecewise directed path, the Lipschitz constant of any map on $\Gamma_{m}$ is the maximum of the Lipschitz constants of the map restricted to directed paths. Thus, by Lemmas \ref{lem:mapintojetspace}(\ref{mapintojetspace2}), \ref{lem:mapintojetspace}(\ref{mapintojetspace5}), and \ref{lem:dccupperbound}, Lip$(F_m) \lesssim \sqrt{m}$. Since diam($\Gamma_m) = 2^{N_m} \leq 2^{Cm\log_2(m+1)+1}$ and $F_m(0_m) = 0$, this implies $F_m(\Gamma_m) \sbs B_{J^{r-1}(\R)}(R')$ with $R' \lesssim (m+1)^{Cm+\frac{1}{2}}$. Then $\delta_{2m}(F_m(\Gamma_m)) \sbs B_{J^{r-1}(\R)}(R'')$ with $R'' \lesssim (m+1)^{Cm+\frac{3}{2}}$. Then $\bar{F}_m(\Gamma_m) = \pi_N(\delta_{2m}(F_m(\Gamma_m))) \sbs B_{J^{r-1}(\R)}(R''+1)$. This implies, for any $R \geq 1$, $R \lesssim (m(R)+1)^{Cm(R)+\frac{3}{2}}$, where the implied constant is independent of $R$. This implies $m(R) \gtrsim \frac{\ln(R)}{\ln(\ln(R))}$ for $R \geq 3$. Plugging this into \eqref{eq:MarkovConstantbnd} yields
$$\Pi_p(B_N(R)) \gtrsim \frac{\ln(R)^{\frac{1}{p}-\frac{1}{2r}}}{\ln(\ln(R))^{\frac{1}{p}+\frac{1}{2r}}}$$
\end{proof}

\bibliographystyle{amsalpha}
\bibliography{markovconvexitypaperbib}

\end{document}